\documentclass{article}

\usepackage[english]{babel}

\usepackage[letterpaper,top=2cm,bottom=2cm,left=3cm,right=3cm,marginparwidth=1.75cm]{geometry}

\usepackage[applemac]{inputenc} 		
\usepackage[T1]{fontenc}    		
\usepackage[colorlinks]{hyperref}      
\usepackage{url}            			
\usepackage{amsthm}
\usepackage{booktabs}       		
\usepackage{amsfonts}       		
\usepackage{amsmath}
\usepackage{wrapfig}
\usepackage{nicefrac}       		
\usepackage{microtype}      		
\usepackage{lipsum}				
\usepackage[square,numbers]{natbib}
\usepackage{mathtools}
\usepackage{algorithm}
\usepackage{algorithmicx}
\usepackage{algpseudocode}
\usepackage{graphicx}
\usepackage[most]{tcolorbox}
\usepackage{multicol,multirow}
\usepackage{indentfirst,latexsym,bm}
\usepackage{amsmath}
\usepackage{subfigure}
\usepackage{amssymb}
\usepackage{xcolor}
\usepackage{comment}
\usepackage{enumitem}
\usepackage{bbm}
\usepackage{tikz}
\usepackage{mdframed}
\usepackage{nicematrix}
\usepackage{bbding}
\usepackage{pifont}

\newcommand{\xmark}{\color{red}\ding{55}}
\newcommand{\cmark}{\color{black}\ding{51}}

\usetikzlibrary{arrows.meta,positioning}
\usepackage{pgfplots}
\pgfplotsset{compat=newest}
\usepgfplotslibrary{fillbetween}
\usetikzlibrary{shapes,decorations}
\usetikzlibrary{fit}

\colorlet{color1}{blue}
\colorlet{color2}{red!50!black}

\definecolor{ivory}{RGB}{218,215,203}

\definecolor{cuhkp}{RGB}{98,56,105} 	
\definecolor{cuhkpl}{RGB}{152,24,147} 	
\definecolor{cuhkb}{RGB}{219,160,1} 	
\definecolor{cuhkbd}{RGB}{178,129,0} 	
\definecolor{cuhkr}{RGB}{88,35,155}  	
\definecolor{blackp}{RGB}{0,0,0} 
\definecolor{redp}{RGB}{255,0,0}
\definecolor{orangep}{RGB}{255,128,0}
\definecolor{brownp}{RGB}{128,77,0}
\definecolor{yellowp}{RGB}{255,230,0}
\definecolor{greenp}{RGB}{128,230,0}
\definecolor{bluep}{RGB}{0,128,255}
\definecolor{purplep}{RGB}{152,24,147}
\definecolor{pinkp}{RGB}{230,0,128}
\definecolor{lavender}{rgb}{0.9, 0.9, 0.98}

\usepackage{hyperref}[6.83]

\hypersetup{
    colorlinks=true, 
    linkcolor=blue!80!black,  
    citecolor=blue!80!black, 
    urlcolor=magenta 
}

\RequirePackage[capitalize,nameinlink]{cleveref}

\crefformat{equation}{\textup{#2(#1)#3}}
\crefrangeformat{equation}{\textup{#3(#1)#4--#5(#2)#6}}
\crefmultiformat{equation}{\textup{#2(#1)#3}}{ and \textup{#2(#1)#3}}
{, \textup{#2(#1)#3}}{, and \textup{#2(#1)#3}}
\crefrangemultiformat{equation}{\textup{#3(#1)#4--#5(#2)#6}}%
{ and \textup{#3(#1)#4--#5(#2)#6}}{, \textup{#3(#1)#4--#5(#2)#6}}{, and \textup{#3(#1)#4--#5(#2)#6}}

\Crefformat{equation}{#2Equation~\textup{(#1)}#3}
\Crefrangeformat{equation}{Equations~\textup{#3(#1)#4--#5(#2)#6}}
\Crefmultiformat{equation}{Equations~\textup{#2(#1)#3}}{ and \textup{#2(#1)#3}}
{, \textup{#2(#1)#3}}{, and \textup{#2(#1)#3}}
\Crefrangemultiformat{equation}{Equations~\textup{#3(#1)#4--#5(#2)#6}}%
{ and \textup{#3(#1)#4--#5(#2)#6}}{, \textup{#3(#1)#4--#5(#2)#6}}{, and \textup{#3(#1)#4--#5(#2)#6}}

\crefdefaultlabelformat{#2\textup{#1}#3}

\theoremstyle{plain}
\newtheorem{theorem}{Theorem}[section]
\newtheorem{lemma}[theorem]{Lemma}

\newtheorem{proposition}[theorem]{Proposition}
\newtheorem{definition}[theorem]{Definition}

\theoremstyle{definition}

\newtheorem{example}[theorem]{Example}
\newtheorem{assumption}[theorem]{Assumption}
\crefname{assumption}{assumption}{assumptions}
\Crefname{assumption}{Assumption}{Assumptions}
\crefname{condition}{condition}{conditions}
\Crefname{condition}{Condition}{Conditions}

\theoremstyle{remark}

\newtheorem*{fact*}{Fact}

\DeclareMathOperator*{\argmin}{argmin}

\newcommand{\cl}{\operatorname{cl}}

\newcommand{\clip}{\mathcal{T}^\delta_\eta}

\newcommand{\cC}{\mathcal C}
\newcommand{\R}{\mathbb{R}}
\newcommand{\N}{\mathbb{N}}
\newcommand{\Rd}{\mathbb{R}^d}

\newcommand{\sL}{L}
\newcommand{\sG}{G}

\newcommand{\sH}{H}

\newcommand{\cO}{\mathcal O}

\newcommand{\cE}{\mathcal E}

\newcommand{\cZ}{\mathcal{Z}}
\newcommand{\cD}{\mathcal{D}_{h}}

\newcommand{\cM}{\mathcal{M}}
\newcommand{\cR}{\mathcal{R}}

\newcommand{\one}{\mathbf{1}}
\newcommand{\bv}{\mathbf{v}}
\newcommand{\bu}{\mathbf{u}}
\newcommand{\bs}{\mathbf{s}}
\newcommand{\bh}{\mathbf{h}}
\newcommand{\bx}{\mathbf{x}}
\newcommand{\by}{\mathbf{y}}
\newcommand{\bz}{\mathbf{z}}
\newcommand{\bff}{\mathbf{f}}

\newcommand{\dom}[1]{\mathrm{dom}(#1)}
\newcommand{\idom}[1]{\mathrm{int\ dom}(#1)}

\newcommand{\iprod}[2]{\langle #1, #2 \rangle}

\title{A New Kernel Regularity Condition for Distributed Mirror Descent: Broader Coverage and Simpler Analysis}
\author{Junwen Qiu\thanks{Industrial Systems Engineering and Management, National University of Singapore
  \\ Email: \{\texttt{jwqiu@nus.edu.sg,ziyangzeng@u.nus.edu,leileimei@u.nus.edu,junyuz@nus.edu.sg}\}} \and 
        Ziyang Zeng${}^*$ \and Leilei Mei${}^*$ \and Junyu Zhang${}^*$ }
\begin{document}
\maketitle
\begin{abstract}
Existing convergence analyses of distributed optimization methods in non-Euclidean geometries typically rely on kernel assumptions: (i) global Lipschitz smoothness and (ii)  bi-convexity of the associated Bregman divergence function. Unfortunately, these conditions are violated by \emph{nearly all} kernels used in practice, leaving a huge theory-practice gap. This work closes this gap by developing a unified analytical tool that guarantees convergence under mild conditions. Specifically, we introduce Hessian relative uniform continuity (HRUC), a regularity satisfied by nearly all standard kernels. Importantly, HRUC is closed under concatenation, positive scaling, composition, and various kernel combinations. Leveraging the geometric structure induced by HRUC, we derive convergence guarantees for mirror-descent-based gradient tracking without imposing any restrictive assumptions. More broadly, our analysis techniques extend seamlessly to other decentralized optimization methods in genuinely non-Euclidean and non-Lipschitz settings.
\end{abstract}


\section{Introduction}
\label{sec:intro}
In this paper, we consider the mirror-descent-based algorithms for the nonconvex distributed optimization problem:
\begin{equation}\label{eq:problem}
    \mathop{\rm minimize}_{x_1,\cdots,x_m\in\,\cZ}\;\; \frac1m\sum_{i=1}^m f_i(x_i)\quad \mbox{s.t.}\quad x_1 = \cdots = x_m.
\end{equation}
Each local objective function $f_i$ is continuously differentiable and possibly nonconvex, and we denote by $f = \frac{1}{m}\sum_i f_i$ the global objective function.
The set $\cZ\subseteq\mathbb{R}
^d$ is closed, convex, and is associated with a \emph{Legendre-type} Bregman distance generating kernel $h$.   
 
For problem \eqref{eq:problem}, the consensus mixing-based first-order methods such as distributed gradient descent ({\sf DGD}) \cite{Tsitsiklis1986Distributed,Nedic2009Distributed,Nedic2015Directed,Yuan2016DGD}, gradient tracking ({\sf GT}) \cite{Qu2018Harnessing}, and their numerous variants \cite{Shi2015EXTRA,Nedic2017DIGing,Scutari2019SONATA,Yuan2019ExactDiffusionI} form a major class of algorithms. Due to the potential to adapt to a problem's local geometry and the ability to handle general relatively smooth (non-Lipschitz smooth) problems \cite{bauschke2017descent,bolte2018first,lu2018relatively}, the mirror descent technique \cite{nemirovsky1983problem} has also been naturally introduced to this area. 
Unfortunately, despite being broadly studied, existing trials to introduce mirror descent to consensus mixing type algorithms typically require two restrictive conditions on kernel $h$: 
\begin{enumerate}[label=\textup{\textrm{(C.\arabic*)}},topsep=1ex,itemsep=1ex,partopsep=0ex,leftmargin=8ex]
\item \label[condition]{C1} \emph{Lipschitz Smoothness:} the mirror map $\nabla h$ is  Lipschitz continuous.  
\item \label[condition]{C2} \emph{Bi-Convexity:} the Bregman divergence $\cD(u,v)$ is convex in $v$. 
\end{enumerate}
\begin{table}[t]
\centering
{\footnotesize
\setlength{\tabcolsep}{5pt}
\renewcommand{\arraystretch}{1.15}
\NiceMatrixOptions{cell-space-limits=2pt}
\begin{NiceTabular}{|p{6cm}| p{1.3cm}|p{.95cm}p{.95cm}|p{4.5cm}|}%
 [ 
   code-before = 
    \rectanglecolor{green!5}{4-3}{5-4}
    \rectanglecolor{red!5}{6-3}{27-3}
    \rectanglecolor{red!5}{6-4}{27-4}
    \rectanglecolor{green!5}{4-5}{27-5}
 ]
\toprule
\Block{3-1}{\textbf{Kernels}} & \Block[c]{3-1}{\textbf{Domain}} &  \Block[c]{1-2}{\textbf{Joint Conditions}} & &  \Block[c]{3-1}{\qquad\; Distortion Modulus \qquad \;$\zeta(\delta)$ of \textbf{HRUC}} \\ \Hline
& & \Block{2-1}{\labelcref{C1}} & \Block{2-1}{\labelcref{C2}}
 & \\
 & & & &    \\[1mm] \Hline
\Block{2-1}{\textbf{Euclidean kernel} \cite{banerjee2005clustering,lan2023policy} \\ $h(x)=\|x\|_A^2,\;A\succ 0$} & \Block{2-1}{$\Rd$} & \Block{2-1}{\cmark} & \Block{2-1}{\cmark} & \Block{2-1}{$0$}    \\
& & & &   \\ \Hline
\Block{2-1}{ \textbf{Boltzmann-Shannon entropy} \cite{banerjee2005clustering,lan2023policy} \\[1mm]$h(x)=\sum_{i=1}^d x_{(i)}\log (x_{(i)}) - x_{(i)}$} & \Block{2-1}{$\Rd_+$} & \Block{2-1}{\xmark} & \Block{2-1}{\cmark} & \Block{2-1}{$\exp(\delta)-1$}   \\
& & & &   \\ \Hline
\Block{2-1}{\textbf{General $h\in{\cal C}^2$} \cite{rodomanov2021greedy}\\[1mm]$\mu$-strongly convex and $\rho$-Lipschitz Hessian 
} & \Block{2-1}{$\Rd$} & \Block{2-1}{\xmark} & \Block{2-1}{\xmark} & \Block{2-1}{$\rho\delta/\mu^2$}     \\
& & & &   \\ \Hline
\Block{2-1}{\textbf{Power kernel} \cite{bolte2018first,zhang2024stochastic} \\[1mm] $h(x)=\frac{\mu}{2}\|x\|^2 + \frac{\|x\|^{r+2}}{r+2},\; r>0$} & \Block{2-1}{$\Rd$} & \Block{2-1}{\xmark} & \Block{2-1}{\xmark} & \Block{2-1}{$\max\{10, r^2 2^{r+1}\}\Big(\frac{\delta}{\mu^{1+1/r}}+\frac{\delta^r}{\mu^{1+r}}\Big)$}     \\
& & & &   \\ \Hline  
\Block{2-1}{\textbf{Regularized Tsallis  entropy} \cite{zimmert2021tsallis,zhan2023policy,lee2018sparse} \\[1mm] $h(x)=\frac{\mu}{2}\|x\|^2 + \frac{1-\sum_{i=1}^d x_{(i)}^q}{1-q},\;q\in(0,1)$} & \Block{2-1}{$\Delta_d$} & \Block{2-1}{\xmark} & \Block{2-1}{\xmark} & \Block{2-1}{$\exp\Big(q^{\frac{q-3}{2-q}}\mu^{\frac{q-1}{2-q}}\cdot \delta\Big)-1$}     \\
& & & &   \\ \Hline
\Block{2-1}{\textbf{Regularized Burg's  entropy} \cite{bauschke2017descent} \\[1mm] $h(x)=\frac{\mu}{2}\|x\|^2 - \sum_{i=1}^d \log(x_{(i)})$} & \Block{2-1}{$\Rd_{++}$} & \Block{2-1}{\xmark} & \Block{2-1}{\xmark} & \Block{2-1}{$\exp(\delta/\sqrt{\mu})-1$}     \\
& & & &   \\ \Hline
\Block{2-1}{\textbf{Exponential  entropy} \cite{raskutti2015information} \\[1mm] $h(x)=\frac{\mu}{2}\|x\|^2 + \sum_{i=1}^d \exp(x_{(i)})$} & \Block{2-1}{$\Rd$} & \Block{2-1}{\xmark} & \Block{2-1}{\xmark} & \Block{2-1}{$\exp(\delta/\mu)-1$}     \\
& & & &   \\ \Hline
\Block{2-1}{\textbf{Norm exponential  kernel} \cite{Nesterov2018} \\[1mm] $h(x)=\exp\Big(\frac{\|x\|^2}{2}\Big)$} & \Block{2-1}{$\Rd$} & \Block{2-1}{\xmark} & \Block{2-1}{\xmark} & \Block{2-1}{$\exp(2\delta)-1$}     \\
& & & &   \\ \Hline
\Block{2-1}{\textbf{Generalized harmonic sum} \cite{denhertog1994inverse} \\[1mm] $h(x)=\frac{\mu}{2}\|x\|^2 + \sum_{i=1}^d x_{(i)}^{-p},\;p>0$} & \Block{2-1}{$\Rd_{++}$} & \Block{2-1}{\xmark} & \Block{2-1}{\xmark} & \Block{2-1}{$\exp\Big(\big[(p+1)^p/(p\mu^{p+1})\big]^{\frac{1}{p+2}}\delta\Big)-1$}     \\
& & & &   \\ \Hline
\Block{2-1}{\textbf{Hellinger function} \cite{tam2023bregman} \\[1mm] $h(x)=-\sum_{i=1}^d \sqrt{1-x_{(i)}^2}$} & \Block{2-1}{$[-1,+1]^d$} & \Block{2-1}{\xmark} & \Block{2-1}{\xmark} & \Block{2-1}{$\exp(3\delta/2)-1$}     \\
& & & &   \\ \Hline 
\Block{2-1}{\textbf{Self-concordant function} \cite{Nesterov2018} \\[1mm] $h$ is $\mu$-strongly convex and $M$-self-concordant} & \Block{2-1}{$\Rd$} & \Block{2-1}{\xmark} & \Block{2-1}{\xmark} & \Block{2-1}{$\exp(2M\delta/\sqrt{\mu})-1$}     \\
& & & &   \\ \Hline 
\bottomrule
\end{NiceTabular}
}
\vspace{1ex}
\caption{
Summary of common kernels, their domains, and which condition they satisfy. Here, we denote $x=[x_{(1)},\ldots,x_{(d)}]^\top \in \Rd$ to avoid confusion with the agent index subscripts. 
Although most kernels violate \ref{C1} and \ref{C2}, they all satisfy the HRUC regularity condition (\Cref{def:kernel cond}) with the associated function $\zeta(\cdot)$. Detailed proofs are provided in \cref{sec:HRUC of kernels}.}

\vspace{-3mm}
\label{table:kernels}
\end{table}

Here, the Bregman divergence is given by $\cD(u,v):=h(u)-h(v)-\langle \nabla h(v),u-v\rangle$. In particular, we call \ref{C2} bi-convexity because $\cD(u,v)$ is already convex in $u$ as the kernel $h$ is always convex in the context of mirror methods. The two conditions \ref{C1} and \ref{C2} are often imposed jointly, see \cite{fang2024gossip,farina2020randomized,li2016distributed,lu2021online,shahrampour2017distributed,suo2025distributed,wu2024stabilized,xiong2023event,yuan2018optimal,yuan2023differentially,zhou2025privacy} and references therein. However, as summarized in \Cref{table:kernels}, most of the kernels violate at least one of these two conditions, except for the classic (Euclidean) quadratic kernel $h(x)=\tfrac12\|x\|^2_A, A\succ 0$, which essentially limits the corresponding analysis to standard distributed algorithms with preconditioning.  This raises a central question: can we provide a theoretically sound approach to introduce the mirror descent techniques to the consensus mixing based distributed first-order algorithms?

Observe that most existing works commonly adopt a primal consensus mixing step in the algorithm design, see \cite[etc.]{li2016distributed,fang2024gossip,farina2020randomized,lu2021online,shahrampour2017distributed,suo2025distributed,wu2024stabilized,xiong2023event,yuan2018optimal,yuan2023differentially,zhou2025privacy}. In these works, given a mixing matrix $W$, an update direction $y_i^t$, and a step size $\eta_t$, the update for each local agent $i$ takes the form:
\begin{equation}   
\label{eq:primal-mixing}
    \nabla h(x_i^{t+1}) = \nabla h\Big({\sum}_{j=1}^m W_{(i,j)}\,x_j^{t}\Big) - \eta_t y_i^t,  
\end{equation}
which follows a mix-and-mirror (primal mixing) scheme. The condition \ref{C2} is essentially introduced to address the non-commutative issue between the nonlinear mirror map $\nabla h(\cdot)$ and the mixing operator $\sum_jW_{(i,j)}(\cdot)$ and, together with \ref{C1}, serves as a technical artifact for decreasing objective value and consensus error. To resolve these issues, we propose the following approaches, which also constitute the major contributions of this paper. \vspace{0.2cm}

\noindent\textbf{New regularity condition. } First, we introduce a new \emph{Hessian relative uniform continuity} (HRUC, \Cref{def:kernel cond}) regularity condition, which, instead of assuming the upper boundedness of $\nabla^2 h$, requires the kernel Hessian  not to explode too rapidly. Informally, HRUC assumes that the kernel Hessians are close in a relative sense:
\begin{equation} 
    \cR_h(x,y):=\Big\|\left[\nabla^2h(y)\right]^{-\frac12}\nabla^2h(x)\left[\nabla^2h(y)\right]^{-\frac12}-I\,\,\Big\|\leq \zeta(\delta) 
\end{equation}
for any $x,y$ that are close in the dual space $\|\nabla h(x)-\nabla h(y)\|\leq\delta$, for any $\delta>0$. The distortion modulus $\zeta$ controls the variation of $\nabla^2h$ and satisfies $\zeta(\delta)\to0$ as $\delta\to0$. Consequently, HRUC implies that the map $x\mapsto\log\!\big(\nabla^2h(x)\big)$\footnote{Here, $\log$ denotes the principal matrix logarithm: for an eigendecomposition $X=Q\operatorname{diag}(\lambda_1,\ldots,\lambda_d)Q^\top$ with $Q^\top Q=I$ and $\lambda_i>0$, $\log(X):=Q\operatorname{diag}(\log\lambda_1,\ldots,\log\lambda_d)Q^\top$ \cite[Chapter~11]{Higham2008}.} is \emph{uniformly continuous} in the dual space, i.e.,
\begin{equation*}
    \left\|\log\!\big(\nabla^2h(x)\big)-\log\!\big(\nabla^2h(y)\big)\right\|
    \leq \log\!\big(1+\zeta(\delta)\big)
    \qquad \text{when}\qquad \|\nabla h(x)-\nabla h(y)\|\leq\delta.
\end{equation*}
In sharp contrast to \ref{C1} and \ref{C2}, HRUC holds for nearly all kernels, as demonstrated in \Cref{table:kernels}. Therefore, HRUC is more of an intrinsic structural property of the mirror methods, rather than a pure technical artifact. Moreover, we show that HRUC is closed under concatenation, composition, and various combination, allowing us to synthesize problem-specific kernels by manipulating standard basic building blocks. 

In addition, we would like to emphasize that the significance of HRUC regularity is beyond the application to distributed optimization in this paper. One can observe that the vanilla (centralized) mirror descent is approximately a preconditioned gradient descent with a time-varying preconditioner: 
\[x^{t+1} = \nabla h^*(\nabla h(x^t)-\eta_t g^t) = x^t - \eta_t[\nabla^2h(x^t)]^{-1}g^t + o(\eta_t\|g^t\|).\]
A guaranteed smooth variation of kernel Hessian (the preconditioner) potentially enables the transformation of many well-developed analysis techniques of gradient methods to mirror descent by viewing the inverse Hessian as variable metrics, which we plan to explore in future works. \vspace{0.2cm}

\noindent\textbf{Dual consensus mixing. } Second, we develop a dual consensus mixing scheme: 
\begin{equation}
    \label{eq:dual-mixing}
    \nabla h(x_i^{t+1})  = {\sum}_{j=1}^m W_{(i,j)}\big( \nabla h(x_j^{t}) - \eta_t y_j^t \big),  
\end{equation}
which, in contrast to the mix-and-mirror update in \eqref{eq:primal-mixing}, follows a reverse mirror-and-mix update order. Interestingly, consider the average iterates in the dual space $\bar{z}^t = \frac{1}{m}\sum_i\nabla h(x_i^t)$ and $\bar{x}^t=\nabla h^*(\bar{z}^t)$, \eqref{eq:dual-mixing} can be viewed as approximately executing the mirror descent on the dual average point: 
$\nabla h(\bar{x}^{t+1}) = \nabla h(\bar{x}^t) - \eta_t\nabla f(\bar{x}^t) + \mathcal{E}_t$, while the error term $\mathcal{E}_t$ diminishes upon consensus. By directly exchanging the order of mirror and mixing operators in \eqref{eq:primal-mixing}, we avoid \ref{C2}. As an illustration, we adopt the {\sf GT} scheme for $y_i^t$ to track the gradient and design a Dual Mixing Gradient Tracking ({\sf DMGT}) method. \vspace{0.2cm}

\noindent\textbf{Broader convergence guarantees. } To broaden the coverage of our theory, we adopt a general relative smoothness assumption \cite{bauschke2017descent,lu2018relatively,bolte2018first} on the objective function, which relaxes the global Lipschitz smoothness and enables analyses of mirror descent for non-Lipschitz smooth objectives~\cite{bauschke2017descent,lu2018relatively,bolte2018first,latafat2022bregman,fatkhullin2024taming,zhang2024stochastic}. Then
leveraging the geometric structure induced by HRUC and dual mixing, we derive an $\cO(1/T)$ convergence for {\sf DMGT} in terms of the Bregman residual and consensus error. By avoiding \ref{C1} and \ref{C2}, our results generally hold for all kernels in \Cref{table:kernels}.
 
\subsection{Extended literature review}
In this subsection, we briefly discuss the related work on decentralized dual averaging ({\sf DDA}) methods~\cite{duchi2011dual} that have not been discussed in the previous section. {\sf DDA} exploits the geometry induced by $h$ via Nesterov's primal-dual (also known as dual averaging) method~\cite{nesterov2009primal}, and updates the local iterates by
\[
\nabla h(x_i^{t+1}) = {\sum}_{j=1}^m \frac{\eta_t W_{(i,j)}}{\eta_{t-1}}\,\nabla h(x_j^{t}) - \eta_t\nabla f_i(x_i^t),
\]
where we have made an equivalent reformulation of \cite{duchi2011dual} to keep a unified formalism with \eqref{eq:primal-mixing} and \eqref{eq:dual-mixing}. It can be observed that, under a constant step size, the {\sf DDA} and dual mixing scheme \eqref{eq:dual-mixing} 
have closely related dual
recursions but differ in whether mixing the local gradients, though neither {\sf DDA} nor dual mixing \eqref{eq:dual-mixing} necessarily takes constant step sizes. In addition, the convergence of {\sf DDA} and its variants is mostly studied under convex and Lipschitz smooth objective functions, e.g., \cite{duchi2011dual,hosseini2013online,tsianos2012consensus}, which also significantly deviate from our focus. To the best of our knowledge, \cite{liu2023rate} is the only analysis of the nonconvex {\sf DDA} method. Yet it is still limited to the Lipschitz smooth setting, and it requires additional strong restrictions on the kernel and initial point, which may not be satisfied by certain kernels such as Burg's entropy, and may contradict the conventional initialization of certain area such as Poisson inverse problem; see details in \cref{subsec:poisson inverse}. 

\subsection{Notations}  By convention, $\|\cdot\|$ represents the $\ell_2$-norm for a vector and the spectral norm for a matrix. $\|\cdot\|_F$ denotes the Frobenius norm for a matrix. Given a positive semidefinite matrix $H\succeq0$ and a vector $u$ with compatible sizes, we denote $\|u\|_H=\sqrt{u^\top H u}$. For a function $f$, we denote its domain by $\dom{f}$. We denote $\cl \dom{f}$, ${\rm bdry}\,\dom{f}$, and $\idom{f}$ the closure, boundary, and interior of $\dom{f}$, respectively. Let $\N:=\{0,1,2,\ldots\}$ denote the natural numbers. For any positive integer $m$, define $[m]:=\{1,\ldots,m\}$. For a vector $x\in\mathbb{R}^d$, we write $x_{(i)}$ for its $i$-th coordinate, $i\in[d]$. For a matrix $X\in\mathbb{R}^{m\times n}$, we write $X_{(i,j)}$ for its $(i,j)$ entry, $i\in[m]$, $j\in[n]$.
Given the collection of vectors $\{x_i\}_{i\in[m]}$, we write $\bx=[x_1^\top,\ldots,x_m^\top]^\top$ for the matrix that stacks $x_i$ vectors.

\section{Preliminaries and basic assumptions}
We briefly introduce a few preliminary concepts, and the basic assumptions used throughout this paper. 

\subsection{Mirror descent and relative smoothness} First, let us formalize the assumption on $\cZ$ and the associated kernel $h$ in the introduction.  
\begin{assumption}
    \label{assump:Legrendre}
    The set $\cZ$ is closed and convex. The associated kernel $h$ is twice continuously differentiable on $\idom{h}$ and \emph{Legendre}, and it satisfies $\mathcal{Z} = \cl\dom{h}$. 
\end{assumption} 
By \emph{Legendre}, we follow the standard definition in \cite{rockafellar1970convex}, which requires the kernel $h$ to be both \emph{essentially smooth} and \emph{essentially strictly convex}. That is, we require $\|\nabla h(x_n)\|\to+\infty$ if $x_n\to x\in{\rm bdry}\, \dom{h}$ and $h$ is strictly convex in $\dom{\nabla h}$. In addition, we make the following smoothness assumptions on the objective function. 

\begin{assumption}
    \label{assump:LowerBound}
    Each local objective function $f_i$, $i\in[m]$, is twice continuously differentiable on $\idom{h}$. The global objective function $f$ is lower bounded: $\inf_{x\in\cZ}f(x)\geq\underline{f}$.
\end{assumption}

\begin{assumption}
    \label{assump:Relative smoothness}
    Each local objective function $f_i$, $i\in[m]$, is smooth relative to $h$ with the modulus ${\sL}>0$ in the sense that both $\sL  h \pm f_i$ are convex on $\idom{h}$.  
\end{assumption}

The relative smoothness condition was originally introduced by \cite{bauschke2017descent,lu2018relatively} to handle objective functions that are not necessarily having globally Lipschitz continuous gradient. Under twice differentiability, it is equivalent to $-\sL \nabla^2 h(x)\preceq \nabla^2 f_i(x)\preceq \sL \nabla^2h(x)$ for any $x\in\idom{h}$. As a direct consequence of \Cref{assump:Relative smoothness}, we know the global objective function $f$ is also smooth relative to $h$ with modulus $\sL$. Based on this, the extended descent lemma \cite{bauschke2017descent,bolte2018first,lu2018relatively} holds for any $(x,y)\in \dom{h}\times \idom{h}$ that 
\begin{equation}\label{eq:extended descent}
\big|f(x) - f(y) - \iprod{\nabla f(y)}{x-y}\big| \leq \sL \cD(x,y), 
\end{equation}
leading to the (centralized) mirror descent for $f$ as a majorization minimization update\vspace{-0.2cm} 
\[
    x^+=\argmin_{z\in\cZ} \, f(x)+\nabla f(x)^\top (z-x) + \frac{1}{\alpha}\, \cD(z,x), \quad \text{with} \quad 0< \alpha \leq \tfrac1\sL. \]
Given \Cref{assump:Legrendre}, the essential smoothness and essential strict convexity imply that $x^+\in\idom{h}$. As a result, $z\in\cZ$ will be inactive in the above subproblem and the update can be written as $\nabla h(x^+) = \nabla h(x) - \alpha\nabla f(x)$.

\begin{example}
    Below, we list several common examples that satisfy \Crefrange{assump:Legrendre}{assump:Relative smoothness}.

\noindent $\bullet$ {(Poisson linear inverse problems \cite{bauschke2017descent})} For each agent $i\in[m]$, let $A_i\in\mathbb R_+^{n\times d}$ have rows $\alpha_{i,j}^\top$, $j\in[n]$, and let $b_i=(b_{i,1},\ldots,b_{i,n})^\top\in\mathbb R_+^n$ be the observation vector. The distributed Poisson linear inverse problem takes the form
    \[
        \min_{x\in\mathcal Z=\mathbb R_+^d}\ f(x):=\frac1m{\sum}_{i=1}^m f_i(x)
        \qquad \text{with} \qquad 
        f_i(x):={\sum}_{j=1}^n\big[\alpha_{i,j}^\top x-b_{i,j}\log (\alpha_{i,j}^\top x)\big].
    \]
    Each $f_i$ is twice continuously differentiable and convex on $\mathbb R_{++}^d$. Moreover, since $t\mapsto t-b\log t$ is bounded below on $t>0$ for every $b\geq0$, each $f_i$ is lower bounded. For any $\mu>0$, consider the regularized Burg's entropy
    $h(x):=\frac{\mu}{2}\|x\|^2-\sum_{k=1}^d\log x_{(k)}$, which is a Legendre kernel with $\cl\dom{h}=\mathcal Z$. By \cite[Lemma 7]{bauschke2017descent}, both $\sL h \pm f_i$ are convex when $\sL\geq\max_{i\in[m]}\|b_i\|_1$ for all $i\in[m]$. Therefore, \Crefrange{assump:Legrendre}{assump:Relative smoothness} hold.

\noindent $\bullet$ {(Quadratic inverse problems \cite{bolte2018first})} For each agent $i\in[m]$, let $A_{i,j}\in\mathbb R^{d\times d}$ be symmetric and $b_{i,j}\in\mathbb R$, $j\in[n]$. The quadratic inverse problem takes the form
    \[
        \min_{x\in\mathcal Z=\Rd}\ f(x):=\frac1m{\sum}_{i=1}^m f_i(x)
        \qquad \text{with} \qquad
        f_i(x):=\frac14{\sum}_{j=1}^n\big(x^\top A_{i,j}x-b_{i,j}\big)^2.
    \]
    Each $f_i$ is twice continuously differentiable and lower bounded. Consider the power kernel $h(x):=\frac14\|x\|^4+\frac12\|x\|^2$, which is a Legendre kernel satisfying $\cl\dom{h}=\mathcal Z$. The Hessian estimate in the proof of \cite[Lemma~5.1]{bolte2018first} shows that 
    \[
        -\sL\nabla^2h(x)\preceq\nabla^2f_i(x)\preceq\sL\nabla^2h(x)
        \quad\text{when}\quad
        \sL\geq\max_{i\in[m]}{\sum}_{j=1}^n
        \big(3\|A_{i,j}\|^2+\|A_{i,j}\||b_{i,j}|\big).
    \]
    Therefore, \Crefrange{assump:Legrendre}{assump:Relative smoothness} hold. Phase retrieval is a special case of this problem if $A_{i,j}=a_{i,j}a_{i,j}^\top$.

\noindent $\bullet$ {(Poisson rate estimation \cite{raskutti2015information})} For each agent $i\in[m]$, let $y_i=(y_{i,1},\ldots,y_{i,d})^\top\in\mathbb N^d$ be a vector of independent Poisson counts. The distributed negative log-likelihood is
    \[
        \min_{x\in \mathcal Z = \Rd}\ f(x):=\frac1m{\sum}_{i=1}^m f_i(x)
        \qquad \text{with} \qquad
        f_i(x):={\sum}_{k=1}^d\big(\exp(x_{(k)})-y_{i,k}x_{(k)}\big).
    \] 
    Each $f_i$ is twice continuously differentiable and convex. Since $t\mapsto \exp(t)-yt$ is bounded below on $\mathbb R$ for every $y\in\R_+$, each $f_i$ is also lower bounded. For any $\mu>0$, consider the regularized exponential kernel $h(x):=\frac{\mu}{2}\|x\|^2+\sum_{k=1}^d \exp(x_{(k)})$, which is a Legendre kernel with $\cl\dom{h}=\mathcal Z$. Moreover, it holds that $ \nabla^2(h-f_i)(x)=\mu I\succeq0$ and
    $$
        \nabla^2(h+f_i)(x)
        =\mu I+2\operatorname{diag}\big(\exp(x_{(1)}),\ldots,\exp(x_{(d)})\big)\succeq0.
    $$
    Therefore, \Crefrange{assump:Legrendre}{assump:Relative smoothness} hold with $\sL=1$.
\end{example}

\subsection{Consensus mixing} Let $W\in\R^{m\times m}_+$ be the mixing matrix, where $W_{(i,j)}>0$ if agent $i$ and $j$ are adjacent nodes in the communication graph. Let $\one \in \R^m$ be the vector of all ones, we make the following assumption on the matrix $W$. 

\begin{assumption}
	\label[assumption]{as:matrix}
    $W$ is symmetric, doubly stochastic, and $\rho:=\|W-\frac1m\one\one^\top\|<1$.  
\end{assumption}
As a consequence, we have $W\one=\one$, $\one^\top W = \one^\top$, and the following  contraction result. 
\begin{lemma}
    \label{lem:contraction}
     Given  \Cref{as:matrix}, and $\{u_i\}_{i\in[m]}$, $\{v_i\}_{i\in[m]}$, $\{v_i^+\}_{i\in[m]} \subset \Rd$.  Let $\bu,\bv,\bv^+\in\R^{m\times d}$ and $\bar u, \bar v, \bar v^+\in\Rd$ be their concatenated matrix and averaged vectors. Let $H,H^+\succ0$ be symmetric positive definite matrices. Assume that 
\[
\bv^+=W\bv+\bu,\qquad \text{and}\qquad \big\|[H^+]^{\frac12} H^{-\frac12}  \big\|^2\leq 1+\alpha\quad \text{with} \quad  \alpha\in\Big[0,\frac{1-\rho}{2}\Big].
\]
Then, 
    \[
    {\sum}_{i=1}^m\|v_i^+ - \bar v^+\|^2_{H^+} \leq \rho \cdot {\sum}_{i=1}^m\|v_i - \bar v\|^2_{H} + \frac{3}{1-\rho}\cdot {\sum}_{i=1}^m \|u_i\|^2_{H}.
    \]
\end{lemma}
For succinctness, we move the proof of the lemma to \cref{appendix: Lemma-contraction}.

\section{The Hessian relative uniform continuity regularity condition}  We invoke a non-Euclidean metric $\rho_h(\cdot\,,\cdot)$ to measure the proximity of points in $\idom{h}$, by incorporating their covectors in the dual space:
\begin{equation*} \rho_h(x,y):=\|\nabla h(x)-\nabla h(y)\|. \end{equation*}
The strict convexity of $h$ indicates that $\rho_h(x,y)=0\Leftrightarrow x=y$ and straightforward computation verifies the nonnegativity, symmetry, and triangle inequality. Thus, $\rho_h(\cdot\,,\cdot)$ indeed induces a distance. Under this distance, we provide the formal definition of the HRUC regularity condition. 

\begin{definition}[\textbf{HRUC Regularity}]
\label{def:kernel cond}
We say a kernel $h$ with positive definite Hessian on $\idom{h}$ is $\zeta$-HRUC regular with respect to a non-decreasing continuous \textbf{distortion modulus} $\zeta\!:\mathbb{R}_+\!\to\!\mathbb{R}_+$ with $\zeta(0)=0$, if 
\begin{equation}
    \label{eqn:HRUC}
  \cR_h(x,y):=  \left\|\big[\nabla^2 h(y)\big]^{-\frac12}\nabla^2 h(x)\big[\nabla^2 h(y)\big]^{-\frac12}-I\right\| \leq \zeta(\delta)
\end{equation}
for all $\delta\geq0$ and $x,y\in\idom{h}$ that satisfy $\rho_h(x,y)\leq \delta$.  
\end{definition}
Since the matrix in \eqref{eqn:HRUC} is symmetric, the Rayleigh theorem \cite[Theorem~4.2.2]{HornJohnson2012} implies
\begin{equation*}
    \cR_h(x,y)
    =\max_{u\neq0}\left|\frac{u^\top\nabla^2h(x)u}{u^\top\nabla^2h(y)u}-1\right|.
\end{equation*}
Thus, HRUC controls relative changes in the kernel Hessian through the distortion modulus $\zeta$. Since $\rho_h$ is symmetric, applying \eqref{eqn:HRUC} to both $(x,y)$ and $(y,x)$ yields the two-sided Hessian comparison:
\begin{equation}
    \label{eqn:HRUC-Hessian-comparison}
    \frac{1}{1+\zeta(\delta)}\nabla^2h(y)\preceq \nabla^2h(x)\preceq \big(1+\zeta(\delta)\big)\nabla^2h(y)\qquad \text{when}\qquad \rho_h(x,y)\leq\delta.
\end{equation} 
Let $A\preceq B$ denote the L\"owner order, i.e., $B-A\succeq0$. The operator monotonicity of the matrix logarithm \cite[Theorem~V.4.7]{Bhatia1997} implies that it preserves the L\"owner order: if $0\prec A\preceq B$, then $\log(A)\preceq\log(B)$. Applying this property to \eqref{eqn:HRUC-Hessian-comparison} and using $\log(cX)=\log(c)I+\log(X)$ for $c>0$ and $X\succ0$ yields
\begin{equation*}
    \left\|\log\!\big(\nabla^2h(x)\big)-\log\!\big(\nabla^2h(y)\big)\right\|
    \leq \log\!\big(1+\zeta(\delta)\big)
    \qquad \text{when}\qquad \rho_h(x,y)\leq\delta.
\end{equation*}
Since $\zeta(\delta)\searrow 0$ as $\delta\searrow 0$, this inequality implies the matrix logarithm $\log \left(\nabla^2 h(\cdot)\right)$ is \emph{uniformly continuous} under the $\rho_h$-distance, which motivates the use of ``uniform continuity'' in the naming of HRUC.  
Geometrically, the corresponding local Hessian metrics differ by at most a factor of $1+\zeta(\delta)$. Note that logarithm can significantly suppress the growth, this reveals the intuition behind the wide applicability of HRUC regularity: if the kernel Hessian does not explode so rapidly that it fails to be uniformly continuous in logarithmic scale, it will possibly satisfy this condition. 

Next, we provide a few important properties for the HRUC-regular kernels.

\begin{proposition}[\textbf{Closedness under Concatenation}]
    \label{proposition:Concat}
    Let the kernel $h$ be \emph{block-separable} in the sense that $h(x) = \sum_{i=1}^mh_i(x_i)$ with $x = (x_{1},\cdots,x_m)$ partitioned into $m$ blocks  $x_i\in\mathbb{R}^{d_i}$. If each $h_i$ is $\zeta_i$-HRUC regular, then the kernel $h$ is $\zeta$-HRUC regular with $\zeta(\cdot):=\max_i\zeta_i(\cdot)$.
\end{proposition}
In practice, many kernels are actually \emph{coordinate-separable} in the sense that each block $x_i\in\mathbb{R}$ is 1-dimensional. The proof of this proposition is straightforward and is hence omitted. Next we present the closedness of HRUC regularity under composition. 
\begin{proposition}[\textbf{Closedness under Composition}]
    \label{proposition:Affine} 
    Let the kernel $h$ be $\zeta$-HRUC regular, and let $A\in\mathbb{R}^{d\times d}$ be a non-singular matrix, $b\in\mathbb{R}^d$ be an arbitrary vector, $c>0$ be a positive scalar, then the kernel $\varphi(\cdot):=ch(A\cdot+b)$ is $\zeta\big(\,\cdot/[c\sigma_{\min}(A)]\big)$-HRUC on its domain, where $\sigma_{\min}(A)>0$ denotes the smallest singular value of $A$.
\end{proposition}
\begin{proof}
    First,  
    we have $\nabla \varphi(x) = cA^\top\nabla h(Ax+b)$ and $\nabla^2\varphi(x) = cA^\top\nabla^2h(Ax+b)A$. For brevity, let $H_x:=\nabla^2h(Ax+b)$ and $H_y:=\nabla^2h(Ay+b)$. Then it is straightforward to check that 
    \begin{equation*}
    \begin{aligned}
        \cR_\varphi(x,y)
        &=\Big\|\big[A^\top H_yA\big]^{-\frac12}A^\top(H_x-H_y)A\big[A^\top H_yA\big]^{-\frac12}\Big\|\\
        &=\Big\|H_y^{-\frac12}(H_x-H_y)H_y^{-\frac12}\Big\|
        =\cR_h(Ax+b,Ay+b),
    \end{aligned}
    \end{equation*}
    where the second equality follows by setting $Q:=H_y^{\frac12}A(A^\top H_yA)^{-\frac12}$, for which $Q^\top Q=I$.
     For any $\delta\geq0$, $\rho_\varphi(x,y)\leq \delta$ indicates that
     \[
     c\sigma_{\min}(A)\rho_h(Ax+b,Ay+b)
     \leq \rho_\varphi(x,y)\leq\delta
     \quad\Longrightarrow\quad
     \rho_h(Ax+b,Ay+b)\leq\frac{\delta}{c\sigma_{\min}(A)}.
     \]
     Together with the $\zeta$-HRUC regularity of $h$, we prove
     $\cR_\varphi(x,y) \leq \zeta\big(\delta/[c\sigma_{\min}(A)]\big)$.
\end{proof}

Next, we discuss the closedness of HRUC regularity under additive combination. 

\begin{proposition}[\textbf{Closedness under Combination}]
    \label{proposition:Affine 1} 
    Let the kernels $h$ and $g$ be $\zeta_h$- and $\zeta_g$-HRUC regular, respectively. Let $\varphi:= h+ g$, then it holds that  \vspace{0.1cm}\\
    \noindent\emph{(i).} $\,\,\,\varphi$ is $\zeta_h$-HRUC if $g(x) = \|x\|^2/2$.\vspace{0.1cm}\\
    \noindent\emph{(ii).} $\,\varphi$ is $\max\{\zeta_h,\zeta_g\}$-HRUC if both $h$ and $g$ are coordinate-separable.\vspace{0.1cm}\\
    \noindent\emph{(iii).} $\!\varphi$ is $\max\{\zeta_h,\zeta_g\}$-HRUC if the kernel pair $(h,g)$ is cross-monotone in the sense that 
    $\langle \nabla h(x)-\nabla h(y),  \nabla g(x)-\nabla g(y)\rangle \geq 0$,  for $x,y\in\dom{\nabla\varphi}$.
\end{proposition}

In this proposition, both (i) and (ii) satisfy the cross-monotone assumption of (iii). However, we state (i) independently because most mirror descent literature requires the kernel $h$ to be strongly convex, see e.g.  \cite{bolte2018first,fatkhullin2024taming,zhang2024stochastic}. If needed, one can always add quadratic terms to the kernel without destroying HRUC. We also state (ii) independently because most of the kernels are indeed coordinate separable. 

\begin{proof}[Proof of \Cref{proposition:Affine 1}] First of all, through direct computation, we obtain 
\begin{equation*} 
\rho_\varphi^2(x,y) = \rho_h^2(x,y) + \rho_g^2(x,y) + 2\langle \nabla h(x)-\nabla h(y),  \nabla g(x)-\nabla g(y)\rangle. \end{equation*}
A direct consequence of the cross-monotonicity is that $\max\left\{\rho_h(x,y),\rho_g(x,y)\right\}\leq \delta$ when $\rho_\varphi(x,y)\leq\delta$. For example, in case (i) where $g$ is quadratic, and in case (ii) where both $h$ and $g$ are coordinate-separable, it is straightforward to verify the cross-monotonicity. Next, let us prove the three cases given $\max\left\{\rho_h(x,y),\rho_g(x,y)\right\}\leq \delta$.\vspace{0.1cm} 

\noindent\textbf{Case (ii).} $\,\,$ In the coordinate-separable case, denote $\nabla^2f(z) = \mathrm{Diag}(f_z^i)\succ0$ for $i\in[d]$, $z\in\{x,y\}$, and $f\in\{h,g\}$.  By cross-monotonicity,   
$\max_i\left\{\left|h_x^i/h_y^i-1\right|\right\}\leq \zeta_h(\delta)$ and  $\max_i\left\{\left|g_x^i/g_y^i-1\right|\right\}\leq \zeta_g(\delta)$. Then direct computation gives
\begin{equation*}
\begin{aligned}
     \cR_\varphi(x,y) &= \max_{i\in[d]}\left\{\left|\frac{h_x^i+g_x^i}{h_y^i+g_y^i}-1\right|\right\}  \\&= \max_{i\in[d]}\left\{\left|\frac{h_y^i(h_x^i/h_y^i-1)}{h_y^i+g_y^i}+\frac{g_y^i(g_x^i/g_y^i-1)}{h_y^i+g_y^i}\right|\right\} \leq \max\{\zeta_h(\delta),\zeta_g(\delta)\}.
\end{aligned}   
\end{equation*}
\noindent\textbf{Case (iii).} $\,\,$ Denote $H_z = \nabla^2h(z)\succ0$ and $G_z = \nabla^2g(z)\succ0$ for $z\in\{x,y\}$. Let $\zeta(\delta):=\max\{\zeta_h(\delta),\zeta_g(\delta)\}$. Applying the Hessian comparison \eqref{eqn:HRUC-Hessian-comparison}, we obtain
\begin{equation*}
\frac{1}{1+\zeta(\delta)}(H_y+G_y)
\preceq H_x+G_x
\preceq\big(1+\zeta(\delta)\big)(H_y+G_y).
\end{equation*}
Therefore, eigenvalues of $(H_y+G_y)^{-\frac12}(H_x+G_x)(H_y+G_y)^{-\frac12}$ lie in $[(1+\zeta(\delta))^{-1},1+\zeta(\delta)]$. Hence,
\begin{equation*}
\begin{aligned}
     \cR_\varphi(x,y)
     &=\left\|(H_y+G_y)^{-\frac12}(H_x+G_x)(H_y+G_y)^{-\frac12}-I\right\|\\
     &\leq \max\left\{\zeta(\delta),1-\frac{1}{1+\zeta(\delta)}\right\}
     =\zeta(\delta).
\end{aligned}
\end{equation*}
\noindent\textbf{Case (i).} $\,\,$ Finally, using the above bound with $\zeta_g(\delta)=0$ (as the quadratic kernel is HRUC with modulus $\zeta_g(\delta)=0$), we obtain $\cR_\varphi(x,y)\leq \zeta_h(\delta)$.
\end{proof}
Combining these closedness propositions, we can easily establish the HRUC regularity of new kernels based on known kernels. 
\begin{example}[Fermi-Dirac entropy]
    Consider the Boltzmann-Shannon entropy
    \begin{equation*} h_{\rm BS}(x) = {\sum}_{i=1}^d\, x_{(i)}\log(x_{(i)})-x_{(i)}. \end{equation*}
    By \cref{table:kernels}, it is $\zeta$-HRUC with $\zeta(\delta) = \exp(\delta)-1$, which will be verified in \Cref{theorem:HURC-List}.  Then for the Fermi-Dirac entropy \cite{bauschke2017descent} defined by:
\begin{equation*} h_{\rm FD}(x):={\sum}_{i=1}^d\, x_{(i)}\log(x_{(i)}) + (1-x_{(i)})\log(1-x_{(i)}), \end{equation*}
we can easily assert that it is also $\zeta$-HRUC with the same $\zeta$ function, which follows from the observation that 
\begin{equation*} h_{\rm FD}(x) = h_{\rm BS}(x) + h_{\rm BS}(-x+\one)+d, \end{equation*} where the constant $d$ can be ignored. By \cref{proposition:Affine}, we assert that $h_{\rm BS}(-x+\one)$ is also $\zeta$-HRUC. Applying \cref{proposition:Affine 1}(ii) then completes the argument.
\end{example}
Next, we present a handy lemma that connects the relative smoothness and the Lipschitz continuity in terms of a local norm and the $\rho_h$-distance in the dual space. 
\begin{lemma}[\textbf{Lipschitz in the distortion geometry}]\label{lem:dual Lipcshitz}
    Let a twice continuously differentiable $f$ be $\sL$-smooth relative to some $\zeta$-HRUC regular kernel $h$. For any covector $z\in\mathrm{Im}(\nabla h)$ and $H_z^*=\nabla^2 h^*(z)$, we have 
    \begin{equation}
        \label{eqn:dual-Lip}
        \| \nabla f(y) - \nabla f(x) \|_{H_z^*} \leq \sL \big(1+\zeta(\delta)\big) \cdot \left\| \nabla h(y) - \nabla h(x) \right\|_{H_z^*} 
    \end{equation}
    as long as the primal vectors $x,y$ satisfy  $\|\nabla h(x)-z\|\leq \delta$ and $\|\nabla h(y)-z\| \leq \delta$.
\end{lemma}

\begin{proof}
    Denote $u =\nabla h(y) - \nabla h(x)$, then we can define the parameterized curve $\gamma:[0,1] \to \Rd$ as $\gamma(q) = \nabla h^*\big(\nabla h(x) + q \cdot u\big)$. Then by Jensen's inequality, we have 
    \begin{equation}\label{eq:distance to curve}
    \left\|\nabla h(\gamma(q)) - z\right\| = \big\|q\big(\nabla h(y) - z\big) + (1-q)\big(\nabla h(x) - z\big) \big\| \leq \delta
    \end{equation}
    for any $q\in[0,1]$. Meanwhile, it follows from the Leibniz integral rule that
    \begin{equation}
        \label{eq:integral expression}
        \nabla f(y) - \nabla f(x) = \!\int_{0}^1 \!\big[\nabla f(\gamma(q)) \big]^\prime   \mathrm{d}q \overset{\text{(i)}}{=} \!\int_0^1 \!\nabla^2 f(\gamma(q))  \big[\nabla^2 h\big(\gamma(q)\big)\big]^{-1} \!  u \; \mathrm{d}q,
    \end{equation}  
    where (i) is due to the chain rule and the fact that $\nabla^2 h^*(\nabla h(w))=[\nabla^2 h(w)]^{-1}$ for any $w\in\idom{h}$. For notational simplicity, let us denote $H^*_z:=\nabla^2 h^*(z)$, $H_{\gamma(q)}:=\nabla^2 h\big(\gamma(q)\big)$ and $F_{\gamma(q)}:=\nabla^2 f\big(\gamma(q)\big)$. 
    Then because $H^*_z$ is symmetric, we have $\|\nabla f(y) - \nabla f(x)\|_{H_z^*} =  \|(H_z^*)^{\frac12}  (\nabla f(y) - \nabla f(x))\|$. Combining \eqref{eq:integral expression}, we have
    \begin{equation}\label{eq:Lipschitz bound 1}
    \begin{aligned}
          \big\|\nabla f(y) \!-\! \nabla f(x)\big\|_{H^*_z} & = \left\|\int_0^1 \!\!(H_z^*)^{\frac12}  H_{\gamma(q)}^{\frac12} \!\cdot H_{\gamma(q)}^{-\frac12} F_{\gamma(q)} \, H_{\gamma(q)}^{-\frac12} \!\cdot H_{\gamma(q)}^{-\frac12}(H_z^*)^{-\frac12}\!\cdot (H_z^*)^{\frac12} u\;\mathrm{d}q \right\|\\
         & \leq \int_0^1 \!\big\| (H_z^*)^{\frac12}  H_{\gamma(q)}^{\frac12}\big\| \cdot \big\| H_{\gamma(q)}^{-\frac12} F_{\gamma(q)} \, H_{\gamma(q)}^{-\frac12}\big\| \cdot \big\|H_{\gamma(q)}^{-\frac12}(H_z^*)^{-\frac12}\big\| \cdot \|u\|_{H_z^*}\; \mathrm{d}q.
    \end{aligned}
      \end{equation}
    Next, we bound each term in \eqref{eq:Lipschitz bound 1}. Because $f$ is $\sL$-smooth relative to $h$, we have 
\[
-\sL I \preceq [\nabla^2 h(w)]^{-\frac12}\, \nabla^2 f(w) \, [\nabla^2 h(w)]^{-\frac12} \preceq \sL I,\quad \forall \, w\in\idom{h}.
\]
Therefore, $\| H_{\gamma(q)}^{-1/2} F_{\gamma(q)} H_{\gamma(q)}^{-1/2} \| \leq \sL$ for all $q \in [0,1]$.  Moreover, by \eqref{eq:distance to curve}, we can apply HRUC regularity and a supporting linear algebra result \Cref{lem:matrix norm bound} to obtain
\[
\big\| (H_z^*)^{1/2} H_{\gamma(q)}^{1/2} \big\| = \big\| [\nabla^2 h(\nabla h^*(z))]^{-1/2} \cdot [\nabla^2 h(\gamma(q))]^{1/2} \big\| \leq \sqrt{1 + \zeta(\delta)} \quad \forall\, q \in [0,1].
\]
Similarly, we can deduce $\| H_{\gamma(q)}^{-1/2} (H_z^*)^{-1/2} \| \leq \sqrt{1 + \zeta(\delta)}$ for all $q \in [0,1]$. Substituting $u = \nabla h(y) - \nabla h(x)$ and the established bounds to \eqref{eq:Lipschitz bound 1} completes the proof. 
\end{proof}

\section{Common kernels are HRUC regular}\label{sec:HRUC of kernels}
In this section, we show the flexibility of HRUC by verifying that representative kernels are HRUC regular.
\begin{theorem}
    \label{theorem:HURC-List} All the kernels listed in \cref{table:kernels} satisfy the HRUC regularity condition, with explicitly computed distortion moduli. 
\end{theorem}
\begin{proof}[Proof sketch]
In \cref{subsec:separable kernel}, we propose sufficient conditions to establish the HRUC regularity of separable kernels, which cover standard instances such as Boltzmann-Shannon, Burg's, and Tsallis entropies, as well as Hellinger functions. Then, in \cref{subsec:Self-concordant}, we demonstrate that self-concordant functions are also HRUC regular. The self-concordant functions class is very broad, see, e.g., \cite[Chapter 5.1]{Nesterov2018}. Finally, in \cref{subsec:power kernel}, we derive sharper HRUC moduli for functions with Lipschitz continuous Hessians and power kernels through direct analysis.
\end{proof}

\subsection{Separable kernels}\label{subsec:separable kernel}
A kernel $h:\dom{h}\to\R$ is called \emph{separable} if it admits the decomposition $h(x)=\sum_{i=1}^d h_i(x_{(i)})$, where $x_{(i)}$ is the $i$-th coordinate of $x\in\Rd$ and $h_i:\dom{h_i}\to\R$. In \Cref{lem:verify kernels}, we establish sufficient conditions for such functions to be HRUC regular. Notably, all separable kernels listed in \cref{table:kernels} satisfy these criteria. Therefore, establishing HRUC regularity and deriving the modulus $\zeta(\delta)$ for a specific separable kernel reduces to a direct verification of the conditions in \Cref{lem:verify kernels}. The proof of \Cref{lem:verify kernels} is in \cref{appendix:separable kernel}. 
\begin{proposition}[\textbf{Sufficient conditions for HRUC regularity}]\label{lem:verify kernels}
    Let $h:\dom{h}\to\R$ be a separable kernel and let $\sG, \sH > 0$ be constants. For each coordinate $i\in[d]$, assume there exists a bijection\footnote{A typical choice is $\phi_i:=h_i^\prime$.} $\phi_i:\idom{h_i}\to\mathrm{Im}(\phi_i)$ satisfying:
       \begin{itemize}
        \item the condition $\rho_h(x,y)\leq \delta$ implies $|\phi_i(x_{(i)})-\phi_i(y_{(i)})|\leq \sH\cdot  \delta$; 
        \item the mapping $g_i(s):=\log\big(h_i^{\prime\prime}(\phi_i^{-1}(s))\big)$ is $\sG$-Lipschitz continuous on $\mathrm{Im}(\phi_i)$.
    \end{itemize} 
    Then, $h$ is HRUC-regular with the modulus $\zeta(\delta):=\exp(\sG\sH  \delta)-1$. 
\end{proposition}
Below, we verify the HRUC regularity of the separable kernels in \cref{table:kernels} via \Cref{lem:verify kernels}.
\smallskip

\noindent $\bullet$ {(Boltzmann-Shannon entropy)} 
Direct computation leads to $h_i^\prime(t)=\log(t)$ and $h_i^{\prime\prime}(t)=\frac{1}{t}$ for each $i\in[d]$.  Define mappings: \[\phi_i(t):=h_i^\prime(t)=\log(t),\;t>0 \quad \text{and} \quad g_i(s):=-s,\; s\in\R. \]
It is easy to verify that $|\phi_i(x_i)-\phi_i(y_i)|\leq\delta$ provided that $\rho_h(x,y)\leq \delta$ and $|g_i^\prime(s)| \leq 1$. Applying \Cref{lem:verify kernels}, we show that HRUC regularity holds with $\zeta(\delta) = \exp(\delta) -1$.\par\smallskip

\noindent $\bullet$ {(Regularized Burg's entropy)} For each $i\in[d]$, direct computation shows  $h_i^\prime(t)=\mu t-1/t$ and $h_i^{\prime\prime}(t)=\mu + 1/t^2$. Define mappings:
\[
\phi_i(t):=1/t,\; t>0 \quad \text{and} \quad g_i(s):=\log(\mu+s^2),\;s>0.
\]
Then, $|g_i^\prime(s)|=|\frac{2s}{\mu+s^2}| \leq 1/\sqrt{\mu}$ and $\rho_h(x,y)\leq \delta$ infers $|\phi_i(x_i)-\phi_i(y_i)|\leq\delta$. By \Cref{lem:verify kernels}, HRUC regularity holds with $\zeta(\delta)=\exp(\delta/\sqrt{\mu})-1$.\par\smallskip

\noindent $\bullet$ {(Regularized Tsallis entropy)} For each $i\in[d]$, we have $h_i^\prime(t)=\mu t- \frac{q}{1-q}\cdot t^{q-1}$ and $h_i^{\prime\prime}(t)=\mu + q \cdot  t^{q-2}$, $q\in(0,1)$ and $t>0$. Introduce mappings
\[
\phi_i(t):=t^{q-1},\;t>0 \quad \text{and} \quad g_i(s):=\log\big(\mu+q\cdot s^{\frac{q-2}{q-1}}\big),\;s>0.
\]
Note that $\rho_h(x,y)\leq \delta$ implies $|\phi_i(x_i)-\phi_i(y_i)| \leq \frac{\delta(1-q)}{q}$ and 
\[g^\prime(s)=\frac{q(2-q)}{1-q}\cdot \frac{s^{\frac{1}{1-q}}}{\mu + q\cdot s^{\frac{2-q}{1-q}}} \leq \frac{2}{2-2q}\cdot \frac{1}{q^{\frac{1}{2-q}}\mu^{\frac{1-q}{2-q}}} \qquad \forall s>0.\] 
By \Cref{lem:verify kernels}, we conclude that HRUC holds with $\zeta(\delta)= \exp\big(q^{\frac{q-3}{2-q}}\mu^{\frac{q-1}{2-q}}\cdot \delta \big)-1$.\par\smallskip
 
\noindent $\bullet$ {(Exponential kernel)} 
Defining $\phi_i(t):=\exp(t)$ and  $g_i(s):=\log(\mu+s)$, $s>0$ and applying \Cref{lem:verify kernels}
, we yield the HRUC function $\zeta(\delta):=\exp(\delta/\mu)-1$.\par\smallskip

\noindent $\bullet$ {(Generalized harmonic sum)} 
Since the proof is almost identical to that of \emph{regularized Tsallis entropy}, we thus omit the proof and conclude that HRUC holds with $\zeta(\delta)=\exp\big(\big[(p+1)^p/(p\mu^{p+1})\big]^{\frac{1}{p+2}}\delta\big)-1$.\par\smallskip

\noindent $\bullet$ {(Hellinger function)} For each $i\in[d]$, we have $h_i^\prime(t)=\frac{t}{\sqrt{1-t^2}}$ and $h_i^{\prime\prime}(t)=(1-t^2)^{-3/2}$, $t\in(-1,1)$. Set
\[
\phi_i(t):=\frac{t}{\sqrt{1-t^2}},\;t\in(-1,1)\quad \text{and} \quad g_i(s):=\frac32\log(1+s^2),\; s\in\R.
\]
Obviously, $\rho_h(x,y)\leq \delta$ implies $|\phi_i(x_i)-\phi_i(y_i)|\leq \delta$. 
Direct calculation gives 
$g^\prime(s)=\frac{3s}{1+s^2}$. Thus, ${\sup}_{s\in\mathbb{R}}\;|g'(s)|=\frac32$.
Applying \Cref{lem:verify kernels}, the HRUC regularity with $\zeta(\delta)=\exp(3\delta/2)-1$.

\subsection{\texorpdfstring{$M$-self-concordant and $\mu$-strongly convex function}
{M-self-concordant and mu-strongly convex function}}\label{subsec:Self-concordant}
Self-concordant functions encompass a broad spectrum of kernels, including strongly convex functions with Lipschitz continuous Hessians,  norm exponential and power kernels. 
The definition of self-concordant function \cite[Definition 5.1.1]{Nesterov2018} is given below.
\begin{definition}[\textbf{Self-concordant function}]
Let $f:\dom f\subseteq\Rd \to\R$ be a closed convex function 
and $f\in \cC^3(\dom f)$. We say that $f$ is \emph{$M$-self-concordant} if 
\[
\big|\nabla^3 f(x)[u,u,u]\big|
\;\leq\;
2M\|u\|^3_{\nabla^2 f(x)},\quad \forall\,x\in\dom f,\; \forall\, u\in\R^{d}.
\]
\end{definition}
Next, we show that self-concordant, strongly convex  functions are HRUC regular. Such a result serves as a useful analytical tool that facilitates the verification and derivation of the modulus $\zeta(\delta)$ for functions in this class.

\begin{proposition}[\textbf{HRUC regularity of self-concordant functions}]\label{prop:HRUC self-concordant}
    Let $M,\mu>0$ be given. Then, $M$-self-concordant and $\mu$-strongly convex functions are HRUC regular with modulus $\zeta(\delta)=\exp(\frac{2M}{\sqrt{\mu}}\delta)-1$.
\end{proposition}
The proof is provided in \cref{appendix:HRUC self-concordant}. Below, we verify that the following kernel classes are self-concordant and hence HRUC regular.
\smallskip

\noindent $\bullet$ {(Norm exponential kernel)}
Consider the kernel $h(x):=\exp(\|x\|^2/2)$, $x\in\Rd$.
Direct calculation yields  $\nabla^2h(x)=h(x)(I+xx^\top)\succeq I$ and
\[
\big|\nabla^3h(x)[u,u,u]\big|
\leq2h(x)\big(\langle x,u\rangle^2+\|u\|^2\big)^{3/2}
\leq2\|u\|_{\nabla^2h(x)}^3\qquad \forall\,u\in\Rd.
\]
Thus, $h$ is $1$-self-concordant and $1$-strongly convex. Invoking \Cref{prop:HRUC self-concordant}, we show that such a kernel is HRUC regular with modulus $\zeta(\delta)=\exp(2\delta)-1$.\par\smallskip

\noindent $\bullet$ {(General $\mu$-strongly convex and $\rho$-Lipschitz Hessian functions)}
Observe that any $\cC^3$ kernel in this class is $M$-self-concordant with $M=\frac{\rho}{2\mu^{3/2}}$.
By \Cref{prop:HRUC self-concordant}, $M$-self-concordant and $\mu$-strongly convex functions are HRUC regular with $\zeta(\delta) = \exp(\rho\delta/{\mu^2})-1$.\par\smallskip

\noindent $\bullet$ {(Power kernels)}
When $r>1$, the kernel $h(x):=\mu\|x\|^2/2+\|x\|^{r+2}/(r+2)$ is $M_r$-self-concordant and $\mu$-strongly convex with $M_r=r(r+3)/(2\mu^{1/2+1/r})$. By \Cref{prop:HRUC self-concordant}, this kernel is HRUC regular with $\zeta(\delta)=\exp\big(r(r+3)\delta/\mu^{1+1/r}\big)-1$.\par\smallskip

\subsection{Sharper HRUC Moduli}\label{subsec:power kernel}
\Cref{prop:HRUC self-concordant} establishes HRUC regularity for all three classes in a simple and unified way. However, this verification imposes stronger regularity conditions and generally does not yield \emph{sharp} moduli. For functions with Lipschitz continuous Hessians, it requires $\cC^3$ smoothness, whereas the direct analysis only requires $\cC^2$ smoothness. For power kernels, it applies only when $r>1$, whereas the direct analysis applies to every $r>0$. In both cases, the self-concordance argument yields exponential moduli $\zeta(\delta)$ that are looser than the direct bounds. This subsection derives sharper moduli under weaker conditions.
\smallskip

\noindent $\bullet$ {(General $\mu$-strongly convex and $\rho$-Lipschitz Hessian functions)} Strong convexity and the $\rho$-Lipschitz continuity of the Hessian imply
\[
\cR_h(x,y)
\leq \frac{1}{\mu}\big\|\nabla^2h(x)-\nabla^2h(y)\big\|
\leq \frac{\rho}{\mu}\|x-y\|
\leq \frac{\rho}{\mu^2}\delta:=\zeta(\delta).
\]
\noindent $\bullet$ {(Power kernels)} In this subsection, we analyze the power kernels
\[h(x)=\frac{\mu}{2}\|x\|^2 + \frac{1}{r+2}\|x\|^{r+2} \quad \text{with $r>0$.}\] We provide the detailed analysis in \cref{appendix:power kernel}, which yields
\[
\hspace{3cm}\cR_h(x,y) \leq \zeta(\delta):=\frac{\max\{10, r^2 2^{r+1}\}}{\mu}\Big(\frac{\delta}{\mu^{1/r}}+\frac{\delta^r}{\mu^r}\Big). \hspace{3cm}
\]

\section{Algorithm and convergence analysis under HRUC}
Having finished the introduction of the new regularity condition, in this section, we illustrate how to design and analyze a distributed optimization algorithm under the HRUC regularity. 
\subsection{Algorithmic design}
\begin{assumption}
    \label[assumption]{assumption:HRUC}
    The kernel $h$ is HRUC-regular  with respect to the function $\zeta$. 
\end{assumption}
For a more compact presentation, we introduce the following notation:
\begin{equation}
	\bx:=\begin{bmatrix}
	    x_1^\top\\[1mm] x_2^\top\\[1mm] \vdots \\[1mm] x_m^\top
	\end{bmatrix}, \qquad
    \bz:=\begin{bmatrix}
	    z_1^\top\\[1mm] z_2^\top\\[1mm] \vdots \\[1mm] z_m^\top
	\end{bmatrix}, \qquad
    \nabla \bh(\bx):=\begin{bmatrix}
	    \nabla h(x_1)^\top\\[1mm] \nabla h(x_2)^\top\\[1mm] \vdots \\[1mm] \nabla h(x_m)^\top
	\end{bmatrix}
	\qquad \text{and} \qquad 
	\nabla \bff(\bx):=\begin{bmatrix}
	\nabla f_1(x_1)^\top\\[1mm]  \nabla f_2(x_2)^\top\\[1mm] \vdots\\[1mm]  \nabla f_m(x_m)^\top
	\end{bmatrix}.
\end{equation}
Denote the dual variable on each agent $i$ and their averaged iterates as 
\[
z_i := \nabla h(x_i)\;\; \forall \,i\in[m]\quad \text{and} \quad \bar z:=\frac{1}{m}{\sum}_{i=1}^m z_i.
\]
In particular, we introduce a clipping operator that operates on a vector $v\in\Rd$:
\begin{equation}\label{eq:clipping operator}
   \clip(v):=\begin{cases}
       \eta v & \; \text{if $\eta\|v\|\leq \delta $},\\[2mm]
        \frac{\delta v}{\|v\|}&\;\text{otherwise}.
   \end{cases} \quad \text{or equivalently } \quad \clip(v):= v \cdot \min\left\{\eta\,,\,\frac{\delta}{\|v\|}\right\}.
\end{equation}
If the input of $\clip$ is a matrix $\bv=[v_1^\top,v_2^\top,\cdots,v_m^\top]^\top \in\R^{m\times d}$, then we default it to the row-wise clipping, that is, $\clip(\bv)^\top = \big[\clip(v_1), \clip(v_2),\cdots, \clip(v_m)\big]$. 
Given these notations, we provide a realization of the dual mixing scheme \eqref{eq:dual-mixing} as \Cref{alg:DMGT}. 

\begin{algorithm}[H]
   \caption{Dual Mixing Gradient Tracking ({\sf DMGT})}
   \label{alg:DMGT}
\begin{algorithmic}
   \State {\bfseries Initialize:} $x_1^0=\cdots=x_m^0$, set $\bz^0=\nabla \bh(\bx^0)$, $\by^0=\nabla \bff(\bx^0)$, kernel $h$, symmetric doubly stochastic mixing matrix $W$, step size $\eta>0$, and clipping threshold $\delta>0$.
   \Repeat\vspace{-4mm}
   \State\begin{minipage}{.95\linewidth}
   \begin{align}
     \label{eq:update z}  \bz^{t+1} &= W \big(\bz^t - \clip(\by^t)\big)\\[1mm]
     \label{eq:update x}\bx^{t+1}&=\nabla \bh^*(\bz^{t+1})\\[1mm]
     \label{eq:update y}  \by^{t+1} &= W \by^t + \nabla \bff(\bx^{t+1}) - \nabla \bff(\bx^{t})
   \end{align}
   \end{minipage}
   \vspace{1mm}
   \Until{meets user specified termination condition.}
\end{algorithmic}
\end{algorithm}
As this algorithm adopts a gradient tracking schedule for the update direction, we call it Dual Mixing Gradient Tracking ({\sf DMGT}) algorithm. In each iteration $t$ of {\sf DMGT}, the agent $i$ maintains and communicates the local variable in its dual form $z_i^t$. Let $\mathcal{N}_i$ be the set of neighboring agents of $i$. Then it performs a clipped gradient step and mixes the neighboring dual local variables  $z_{i}^{t+1} = \sum_{j\in\mathcal{N}_i}W_{(i,j)}(z_{j}^t-\clip(y_j^t))$, obtains the primal local variables by mapping the dual variable  back to primal space $x_{i}^{t+1} = \nabla h^*(z_i^{t+1})$, and it tracks the global gradient information in the style of \cite{Qu2018Harnessing} through $y_{i}^{t+1} = \sum_{j\in\mathcal{N}_i} W_{(i,j)}y_{j}^t + \nabla f_i(x_i^{t+1})-\nabla f_i(x_i^{t})$. 

Compared to the conceptual dual mixing scheme \eqref{eq:dual-mixing} in the introduction, a minimal modification in \eqref{eq:update z} is the adoption of a clipping operator $\clip(\cdot)$. This step is incorporated to secure a relatively smooth change of the kernel Hessians by invoking HRUC regularity. In detail, denote $\bs^t:=\clip(\by^t)$ with  rows $s_i^t:= \clip(y^t_i)$, then $\|s_i^t\|\leq \delta$, equation \eqref{eq:update z} gives $\bz^{t+1} = W (\bz^t - \bs^t)$, and moreover,
\begin{equation}\label{eq:update bar z}
\bar z^{t+1} = \bar z^t - \bar s^t
\quad \text{where} \quad \bar s^t:=\frac1m{\sum}_{i=1}^m \,s_i^t\quad \text{and} \quad \bar z^t:=\frac1m{\sum}_{i=1}^m \,z_i^t.    
\end{equation}

\paragraph{Removing the bi-convexity requirement via dual mixing}
The dual mixing update \eqref{eq:update z} is linear in the dual space. Together with the column stochasticity of $W$, this yields the recursion \eqref{eq:update bar z} for averaged iterates. By contrast, the primal mixing update \eqref{eq:primal-mixing} first forms the weighted average $\tilde x_i^t:=\sum_{j=1}^mW_{(i,j)}x_j^t$ in the primal space and then maps it to the dual space through the \emph{nonlinear} mirror map $\nabla h$. Averaging this update over the agents yields
\[
\bar z^{t+1}
=\frac1m{\sum}_{i=1}^m\nabla h\Big({\sum}_{j=1}^mW_{(i,j)}x_j^t\Big)
-\frac{\eta_t}{m}{\sum}_{i=1}^m y_i^t,
\]
whose nonlinear first term generally cannot be reduced to $\bar z^t=\frac1m\sum_j\nabla h(x_j^t)$ because $\nabla h$ does not commute with averaging. Thus, rather than working with a closed recursion for averaged iterates such as \eqref{eq:update bar z}, existing primal mixing analyses \cite{li2016distributed,yuan2018optimal,fang2024gossip} control the effect of mixing through the aggregate Bregman divergence. Specifically, for any $u\in\dom{h}$, they use
\[
{\sum}_{i=1}^m\cD(u,\tilde x_i^t)
\overset{\text{(i)}}{\leq} {\sum}_{i=1}^m{\sum}_{j=1}^mW_{(i,j)}\cD(u,x_j^t)
= {\sum}_{j=1}^m\cD(u,x_j^t).
\]
This bound ensures that primal mixing does not increase the aggregate Bregman divergence. However, Jensen's inequality (i) requires convexity of $v\mapsto\cD(u,v)$, namely, \ref{C2}. Dual mixing avoids this Jensen step because it works directly with the closed averaged recursion \eqref{eq:update bar z}.
 
\paragraph{Removing the global Lipschitz smoothness via clipping and HRUC}
The condition \ref{C1} imposes a uniform upper bound on $\nabla^2h$. By contrast, clipping \eqref{eq:clipping operator} keeps the iterates within controlled dual-space neighborhoods, where HRUC provides relative Hessian control. Thus, the analysis only requires relative control along the iterates, rather than a global upper bound on $\nabla^2h$. The following lemma quantifies these clipping-induced bounds.

\begin{lemma}
\label{lem:iterates bounds}
Suppose \Cref{as:matrix} holds, and let $\{\bx^t\}$, $\{\by^t\}$, $\{\bz^t\}$ be generated by \Cref{alg:DMGT}. Then, it holds for all $t\in\N$ and $i\in[m]$ that
 \[
    \|\bar z^t - \bar z^{t+1}\| \leq \delta,\quad \|\bz^{t} - \one (\bar z^{t})^\top\|_F \leq \frac{\sqrt{3m}}{1-\rho}\cdot \delta, \quad \|\bz^t - \bz^{t+1}\|_F\leq \frac{\sqrt{20m}}{1-\rho}\cdot \delta. 
\]
\end{lemma} 
The proof of \Cref{lem:iterates bounds} resembles that of \cite[Lemma 5.4]{hong2022divergence}, and is omitted. The bounds in \Cref{lem:iterates bounds} keep consecutive local iterates close in the dual space, ensuring that the HRUC bound \eqref{eqn:HRUC} and the dual-space Lipschitz estimate \eqref{eqn:dual-Lip} apply along the iterates, with constants explicitly determined by $\zeta$. The clipping operator returns the ordinary step $\clip(y_i^t)=\eta y_i^t$ whenever $\eta\|y_i^t\|\leq\delta$. Later, \Cref{prop:finite many clip} shows that clipping occurs only finitely many times.  
 
\subsection{Lyapunov analysis and convergence rate}
Having clarified the algorithmic design, we illustrate how the analysis of standard distributed algorithms extends seamlessly to the mirror setting through the lens of HRUC regularity. Following the discussion of \eqref{eq:dual-mixing} in the introduction, we define the ``average'' primal variable by $\bar x^t:=\nabla h^*(\bar z^t)$. For brevity, we introduce a quantity
\begin{equation}
\label{defn:lambda}
\lambda:=1+\zeta\Big(\tfrac{\sqrt{20m}}{1-\rho}\cdot \delta\Big),
\end{equation}
which will be used throughout the latter discussion. The convergence analysis proceeds as follows. First, \Cref{lem:first descent,lem:consensus errors} establish an objective-descent estimate and a contraction recursion for consensus errors, respectively. Combining these estimates, \Cref{prop:descent} yields a Lyapunov descent inequality, from which \Cref{thm:complexity} derives the $\cO(1/T)$ rate for the optimality measure.  

\begin{lemma}
\label{lem:first descent}
Given \Crefrange{assump:Legrendre}{as:matrix} and \Cref{assumption:HRUC}, for any iteration  $t\in\N$ of \Cref{alg:DMGT}, there exists $\theta_t\in[0,1]$ such that 
for all $\alpha_1,\alpha_2>0$, the following holds: 
    \begin{equation*} \begin{aligned}
    f(\bar x^{t+1})  - f(\bar x^t)\leq & - \sL \cdot \cD(\bar x^t,\bar x^{t+1}) -\Big[\frac{1}{\eta}-\frac{2\sL+\alpha_1+\alpha_2}{2}\Big]\cdot \frac{1}{m}{\sum}_{i=1}^m\|s^t_i\|_{H_{\theta_t}^*}^2 \\& + \frac{1}{2m \alpha_1} {\sum}_{i=1}^m\|y_i^t - \bar y^t\|^2_{H_{\theta_t}^*} + \frac{\sL^2\lambda^2}{2m\alpha_2} {\sum}_{i=1}^m \|\bar z^t - z_i^t\|_{H_{\theta_t}^*}^2,
    \end{aligned} \end{equation*}
    where $H_{\theta_t}^*:=\nabla^2 h^*(\bar{z}_{\theta_t})\succ0$, with $\bar{z}_{\theta_t}=(1-\theta_t) \bar z^t + \theta_t \,\bar z^{t+1}$.
\end{lemma}
This lemma captures the iterative descent (with controllable errors) of {\sf DMGT} using the local norms induced by the kernel's (conjugation's) Hessian, which satisfies the relationship that $\nabla^2 h^*(z) = [\nabla^2 h(x)]^{-1}$ for all $z=\nabla h(x)$. The proof of this lemma is relegated to \cref{proof of lem:first descent}. In the next lemma, we capture the iterative descent and how the consensus error evolves for {\sf DMGT}.  
\begin{lemma} 
\label{lem:consensus errors}
Given \Crefrange{assump:Legrendre}{as:matrix}, \Cref{assumption:HRUC}, and define the consensus potential  
\begin{equation}
    \label{def:E_t}
    \cE_t:=\frac1m{\sum}_{i=1}^m \Big(\|y_i^t - \bar y^t\|^2_{H_{\theta_t}^*} + \xi \cdot \|\bar z^{t} - z_i^{t}\|_{H_{\theta_{t}}^*}^2\Big),
\end{equation}  
with the constant $\xi = \frac{32\sL^2\lambda^2}{(1-\rho)^2}$. Then, the following bound is valid for all $t\in\N:$
\begin{equation}
    \label{eq:consensus E} 
    \cE_{t+1} \leq \frac{1+\rho}{2}\cdot \cE_t + \frac{108\sL^2\lambda^2}{(1-\rho)^{3}}\cdot \frac1m{\sum}_{i=1}^m\|s^t_i\|_{H_{\theta_t}^*}^2,
\end{equation} 
as long as we select the clipping radius $\delta$ such that $\zeta(2\delta)\leq \frac{1-\rho}{2}\in(0,1).$
\end{lemma}

A key difference to standard analysis is the adoption of the local norms $\|\cdot\|_{H_{\theta_t}^*}$, which are introduced to handle the non-homogeneous geometry induced by the  nonlinear mirror map $\nabla h$. By the clipping bound (\Cref{lem:iterates bounds}), we know $\|\bar{z}_{\theta_{t+1}}-\bar{z}_{\theta_t}\|\leq 2\delta$, which, together with the $\zeta$-HRUC regularity of the kernel $h$ (\Cref{assumption:HRUC}), gives 
\begin{equation}
    \label{eqn:H-rel-diff}
    \Big\|[H_{\theta_{t+1}}^*]^{\frac{1}{2}}[H_{\theta_t}^*]^{-\frac{1}{2}}\Big\|^2
    =\Big\|[H_{\theta_t}^*]^{-\frac{1}{2}}H_{\theta_{t+1}}^*[H_{\theta_t}^*]^{-\frac{1}{2}}\Big\|
    \leq 1+\zeta(2\delta),
\end{equation}
It provides a smooth transition between the consecutive local norms: 
$$\frac{1}{1+\zeta(2\delta)}\leq \frac{\|\cdot\|_{H_{\theta_{t+1}}^*}}{\|\cdot\|_{H_{\theta_t}^*}\,\,\,\,\,} \leq    1+\zeta(2\delta),$$ which is sufficient for a contraction when  $\delta$ is small enough such that $\zeta(2\delta)\leq \frac{1-\rho}{2}\in(0,1)$. For succinctness, we move the rest of proof to 
  \cref{proof of lem:consensus errors}.

In view of \Cref{lem:first descent,lem:consensus errors}, we introduce the following potential sequence 
$$\cM_{t}:=f(\bar x^{t}) + \frac{1}{8\sL}\cdot \cE_t.$$
The next lemma establishes the monotonicity of this sequence.  

\begin{lemma}\label{prop:descent}
Given the setting of \Cref{lem:consensus errors},  and select $0<\eta \leq \frac{(1-\rho)^{3}}{25\sL\lambda^2}$, then    
\begin{equation}\label{eq:prop:descent}
\cM_{t+1}\leq \cM_t - \frac{1}{12\eta}\cdot \frac{1}{m}{\sum}_{i=1}^m\|s^t_i\|_{H_{\theta_t}^*}^2 - \sL  \cD(\bar x^t,\bar x^{t+1}) -\frac{(1-\rho)}{32 \sL}\cdot \cE_t  \qquad \forall\;t\in\N.
\end{equation}
\end{lemma}

The proof is in \cref{appendix:proof of prop:descent}.
As a direct consequence of \eqref{eq:prop:descent}, the following bound holds:
\begin{equation*} \begin{aligned}
    \frac{1}{T}\sum_{t=0}^{T-1} \bigg[\frac{1}{12\eta}\cdot \frac{1}{m}{\sum}_{i=1}^m\|s^t_i\|_{H_{\theta_t}^*}^2 + \sL  \cD(\bar x^t,\bar x^{t+1}) + \frac{(1-\rho)}{32 \sL}\cdot \cE_t \bigg] \leq \frac{\cM_0- \cM_T}{T}.
\end{aligned} \end{equation*}
Note $\|\bar{z}^{t+1}-\bar{z}^t\|^2_{H_{\theta_t}^*} = \|\bar{s}^t\|^2_{H_{\theta_t}^*}\leq\frac{1}{m}\sum_{i=1}^m\|s^t_i\|_{H_{\theta_t}^*}^2$, we can define an optimality measure 
\begin{equation*} \mathcal{G}(\bx^t,\by^t,\bz^t):=\frac{1}{12\eta^2}\|\bar{z}^{t+1}-\bar{z}^t\|^2_{H_{\theta_t}^*} + \frac{\sL}{\eta}\cD(\bar x^t,\bar x^{t+1}) + \frac{(1-\rho)}{32 \sL\eta}\cdot \cE_t, \end{equation*}
which is a combination of the Bregman residual of primal iterates $\cD(\bar x^t,\bar x^{t+1})$, the local normed dual residual $\|\bar{z}^{t+1}-\bar{z}^t\|^2_{H_{\theta_t}^*}$, and the consensus error $\cE_t$. With the fact that $\cM_0 = f(x^0) + \cE_0/(8\sL)$ and $\cM_T \geq \underline{f}$, we obtain the following theorem.  

\begin{theorem}
\label{thm:complexity}
Given \Crefrange{assump:Legrendre}{as:matrix} and \Cref{assumption:HRUC}, select the step size $0<\eta \leq \frac{(1-\rho)^{3}}{25\sL\lambda^2}$, and select the clipping radius $\delta$ such that  $\zeta(2\delta)\leq\frac{1-\rho}{2}$,  then
\begin{equation*}
    \min_{t=0,\ldots,T-1}\; \mathcal{G}(\bx^t,\by^t,\bz^t) \leq  \frac{f(x^0)- \underline{f}{}+\cE_0/(8\sL)}{T\cdot\eta}  = \cO\Big(\frac{1}{\eta T}\Big) 
\end{equation*} 
holds for the iterates generated by \cref{alg:DMGT}, for any positive integer $T\geq 1$.
\end{theorem}

\subsection{Stationarity and consensus bounds}\label{subsec:Stationarity and consensus bounds}
To interpret this bound, we consider two settings: (i) $h$ is the standard quadratic kernel $h(x)=\frac{1}{2}\|x\|^2$, for which the local geometry is Euclidean and \Cref{alg:DMGT} reduces to a standard gradient tracking scheme; and (ii) $h$ is a general HRUC kernel for which the dual images of the objective sublevel sets are bounded. \vspace{0.25cm} 

\noindent\textbf{Case I: Quadratic kernel.} By \cref{theorem:HURC-List} and \cref{table:kernels}, the kernel $h(x) = \frac{1}{2}\|x\|^2$ is $\zeta$-HRUC regular with $\zeta(\delta)\equiv0$. In this situation, $\cZ=\dom{h} = \mathbb{R}^d$, and we can set $\delta=+\infty$, then the clipping operator $\clip(v) \equiv \eta v$ is always inactive (cf. \eqref{eq:clipping operator}), and the constant $\lambda = 1$ (cf. \eqref{defn:lambda}). In addition, we also have $\cD(\bar{x}^t,\bar{x}^{t+1}) = \frac{1}{2}\|\bar{x}^t-\bar{x}^{t+1}\|^2$, $H^*_{\theta_t}\equiv I$, and $\bx^t = \bz^t$. As a consequence,  
\begin{equation}
\label{eq:cG-Quadratic}
\mathcal{G}(\bx^t,\by^t,\bz^t) = \left(\frac{1}{12}+\frac{\sL\eta}{2}\right)\|\bar{y}^t\|^2 + \frac{1-\rho}{32\sL\eta}\cdot\frac{1}{m}{\sum}_{i=1}^m\left(\|y_i^t-\bar{y}^t\|^2 + \xi\|x_i^t-\bar{x}^t\|^2\right),    
\end{equation} 
where $\xi = \frac{32\sL^2\lambda^2}{(1-\rho)^2}$ is defined in \Cref{lem:consensus errors}. Then, the following inequality holds: 
\begin{equation}\label{eq:Euclidean s bound}
\left\|\nabla f(\bar x^t)\right\|^2 
    \leq 2\, \Big\|\frac{1}{m}{\sum}_{i=1}^m (\nabla f(\bar x^t) - y^t_i)\Big\|^2 + 2\|\bar{y}^t\|^2 \leq \frac{2}{m}{\sum}_{i=1}^m \left\|\nabla f(\bar x^t) - y^t_i\right\|^2 + 2\|\bar{y}^t\|^2.
\end{equation} 
Moreover, as $\bar y^t=\frac{1}{m}\sum_{j=1}^m \nabla f_j(x_j^t)$ always holds, we have 
\begin{equation}\label{eq:Euclidean gradient bound}
\begin{aligned}
\|\nabla f(\bar x^t) - y_i^t\|^2 
    & =  \Big\|(\bar y^t - y_i^t)+ \frac1m{\sum}_{j=1}^m \big(\nabla f_j(\bar x^t) - \nabla f_j(x_j^t) \big)  \Big\|^2 \\
    & \leq  2\| \bar y^t - y_i^t \|^2 + \frac{2\sL^2}{m}{\sum}_{j=1}^m \|\bar x^t-x_j^t\|^2. 
\end{aligned}
\end{equation}
Then combining the inequalities \eqref{eq:cG-Quadratic}-\eqref{eq:Euclidean gradient bound} and \Cref{thm:complexity} indicates that 
$$\min_{0\leq t\leq T-1}\left\{\|\nabla f(\bar{x}^{t})\|^2 + \frac{\sL^2}{m}{\sum}_{i=1}^m\|\bar{x}^t-x_i^t\|^2\right\} = \cO\left(\frac{\sL(f(x^0)- \underline{f}{}+\cE_0/(8\sL))}{(1-\rho)^{3}T}\right) \qquad\forall\; T\geq 1.$$

\noindent\textbf{Case II: General HRUC kernel.} For a general Legendre-type HRUC kernel, we impose the following boundedness condition on the dual images of the objective sublevel sets.
\begin{assumption}\label[assumption]{as:bounded level}
    For any given $c\in\R$, the set $\{\nabla h(x):f(x) \leq c\}$ is bounded. 
\end{assumption}
This typically holds when $f$ explodes on $\mathrm{bdry\,}\dom{h}$. Valid examples include the D-optimal design problem \cite{atwood1969optimal,kiefer1960equivalence,lu2018relatively}, and various volumetric optimization problems \cite{anstreicher2000volumetric,lu2018relatively}, associated with the Burg's  entropy kernel, the quadratic inverse problem \cite{bolte2018first}, and more general problems with coercivity, associated with polynomial or power kernels, etc. As a consequence, define the set of all clipped iteration indices over the full trajectory as 
$${\cal T}:=\{t\in\N:\exists \,i\in[m] \text{ such that } \eta y_i^t \neq s_i^t\}.$$ 
The next proposition shows that there exists a finite constant $T_{\rm clip}<+\infty$ that only depends on the problem, kernel, and initialization, such that  $|{\cal T}|\leq T_{\rm clip}$, whose detailed definition and proof can be found in \cref{appendix:proof of prop:finite many clip}. Therefore, within any sufficiently long finite running time $T$, most updates in \Cref{alg:DMGT} remain unclipped, allowing a simple interpretation of \Cref{thm:complexity}. 
\begin{proposition}\label{prop:finite many clip}
    Under the setting of \Cref{thm:complexity}, suppose \Cref{as:bounded level} holds in addition. Then the clipping in \Cref{alg:DMGT} happens at most finitely many times.
\end{proposition}

\paragraph{Local-norm stationarity and dual consensus}
Given the finite bound on $|\mathcal{T}|$, and suppose \cref{alg:DMGT} runs for at least $T\geq 2|\mathcal{T}|$ iterations. Let us denote $\mathcal{T}^c:=\{0,\cdots,T-1\}\backslash\mathcal{T}$ the unclipped steps. Then following \eqref{eq:Euclidean s bound} and \eqref{eq:Euclidean gradient bound}, the next inequality holds for all $t\in\mathcal{T}^c$:
\begin{equation*} 
\left\|\nabla f(\bar x^t)\right\|^2_{H_{\theta_t}^*} \leq \frac{4}{m}{\sum}_{i=1}^m\| \bar y^t - y_i^t \|^2_{H_{\theta_t}^*} + \frac{4}{m}{\sum}_{i=1}^m\left\|\nabla f_i(\bar x^t) - \nabla f_i(x^t_i)\right\|^2_{H_{\theta_t}^*} + 2\|\bar{y}^t\|^2_{H_{\theta_t}^*}. 
\end{equation*}
Note that the clipping step (\Cref{lem:iterates bounds}) guarantees that the dual covector $\bar{z}_{\theta_t}:=(1-\theta_t)\bar{z}^t+\theta_t\bar{z}^{t+1}$ that defines the matrix $H^*_{\theta_t}$ satisfies $\|\bar{z}_{\theta_t}-\bar{z}^t\|\leq \delta$, and $\|z_i^{t}-\bar{z}^t\|\leq\frac{\sqrt{3m}}{1-\rho}\cdot\delta$. Then the dual Lipschitz bound in \Cref{lem:dual Lipcshitz} holds with constant $\zeta\big((1+\frac{\sqrt{3m}}{1-\rho})\delta\big)$:   
\begin{equation*} 
    \| \nabla f_i(\bar{x}^t) - \nabla f_i(x_i^t) \|_{H^*_{\theta_t}} \leq \sL \Big[1+\zeta\big((1+\tfrac{\sqrt{3m}}{1-\rho})\delta\big)\Big] \cdot \left\| \bar{z}^t - z_i^t \right\|_{H^*_{\theta_t}}\leq \sL \lambda \left\| \bar{z}^t - z_i^t \right\|_{H^*_{\theta_t}}. 
\end{equation*}
Combining the above two inequalities, and the fact that  $\bar{s}^t = \eta\bar{y}^t$ for $t\in\mathcal{T}^c$, we obtain  
\begin{equation*}
\begin{aligned}
    & \min_{0\leq t\leq T-1} \!\bigg\{\!\|\nabla f(\bar{x}^{t})\|^2_{H^*_{\theta_t}} + \frac{\sL^2}{m}{\sum}_{i=1}^m\|\bar{z}^t-z_i^t\|^2_{H^*_{\theta_t}}\!\bigg\} \leq \min_{t\in\mathcal{T}^c} \left\{\!\|\nabla f(\bar{x}^{t})\|^2_{H^*_{\theta_t}} + \frac{\sL^2}{m}{\sum}_{i=1}^m\|\bar{z}^t-z_i^t\|^2_{H^*_{\theta_t}}\right\}\\
& \lesssim \frac{1}{T-|\mathcal{T}|}\sum_{t\in\mathcal{T}^c}  \mathcal{G}(\bx^t,\by^t,\bz^t) \leq  \cO\left(\frac{\sL(f(x^0)- \underline{f}{}+\cE_0/(8\sL))}{(1-\rho)^{3}T}\right) \quad\mbox{for all}\quad T\geq 2|\mathcal{T}|,    
\end{aligned}
\end{equation*}
where $\lesssim$ hides the numerical constants. The gradient size is measured by its local norm $\|\cdot\|_{H^*_{\theta_t}}$ and the consensus is measured in the dual space for $\bar{z}^t-z_i^t = \nabla h(\bar{x}^t)-\nabla h(x_i^t)$. To connect this local-norm stationarity and dual consensus bound to the original problem, we next convert it into a bound on the Euclidean gradient norm and primal consensus error.

\paragraph{Euclidean stationarity and primal consensus}
The boundedness of the dual images of objective sublevel sets in \Cref{as:bounded level}, together with the Lyapunov descent in \Cref{prop:descent}, ensures that $\{\bar z^t\}$ is bounded. The dual-disagreement bound $\|z_i^t-\bar z^t\|\leq\frac{\sqrt{3m}}{1-\rho}\delta$ from \Cref{lem:iterates bounds} then keeps every $z_i^t$ uniformly close to $\bar z^t$. Hence, the line segments joining $\bar z^t$ and $z_i^t$ remain in a bounded dual region. Moreover, due to $\nabla^2h^*(z)=[\nabla^2h(\nabla h^*(z))]^{-1}$ for all $z\in\mathrm{Im}(\nabla h)$, the HRUC comparison \eqref{eqn:HRUC-Hessian-comparison} implies that there exists $0<L^*<\infty$ such that 
\[
\nabla^2h^*\big((1-\tau)z_i^t+\tau\bar z^t\big)\preceq L^*I,\qquad \forall\,t\in\N,\ i\in[m],\ \tau\in[0,1].
\]
Together with $H_{\theta_t}^*\succeq\mu^*I$ (see \eqref{eq:prop Hessian bound}), this yields $\|\nabla f(\bar x^t)\|^2
\leq \frac{1}{\mu^*}\|\nabla f(\bar x^t)\|_{H_{\theta_t}^*}^2$ and 
\[
\|\bar x^t-x_i^t\|^2
=\left\|\int_0^1\nabla^2h^*\big((1-\tau)z_i^t+\tau\bar z^t\big)(\bar z^t-z_i^t)\,\mathrm d\tau\right\|^2\leq (L^*)^2\|\bar z^t-z_i^t\|^2\leq \frac{(L^*)^2}{\mu^*}\|\bar z^t-z_i^t\|_{H_{\theta_t}^*}^2.
\]
Substituting these estimates into the preceding bound yields
\[
\min_{0\leq t\leq T-1}\left\{\|\nabla f(\bar x^t)\|^2+\frac{\sL^2}{m}{\sum}_{i=1}^m\|\bar x^t-x_i^t\|^2\right\}=\cO\left(\frac{\sL(f(x^0)-\underline f+\cE_0/(8\sL))}{(1-\rho)^3T}\right),\qquad T\geq2|\mathcal T|.
\] 
\section{Numerical experiments}\label{sec:experiment}
We test the practical performance of our proposed {\sf DMGT} method (see \Cref{alg:DMGT}) 
on phase retrieval and Poisson inverse problems. All experiments are run in \texttt{MATLAB} R2024b on a macOS system with M1 Pro processor and 16 GB of memory. 
\paragraph{Baseline algorithms}
We compare {\sf DMGT} with decentralized baselines from the mirror descent and dual averaging literature.
We include {\sf DMD}, the distributed mirror descent method of \cite{li2016distributed}.
We also include {\sf DDA}, the distributed dual averaging method with gradient tracking ({\sf GT}) \cite{liu2023rate}, building on \cite{duchi2011dual,Qu2018Harnessing}.
Since both {\sf DMGT} and {\sf DDA} incorporate GT, for fair comparison with {\sf DMD}, we additionally consider a {\sf GT} variant of {\sf DMD}, denoted {\sf DGT}:
$$
\bz^{t+1} = \nabla \bh\big(W\bx^t\big) - \eta\by^t,\qquad
\bx^{t+1} = \nabla \bh^*(\bz^{t+1}),\qquad
\by^{t+1} = W\by^t + \nabla \bff(\bx^{t+1}) - \nabla \bff(\bx^t).
$$
In all experiments, agents communicate over a network encoded by a mixing matrix $W\in\mathbb{R}^{m\times m}$ associated with an Erd\H{o}s--R\'enyi random graph.

\paragraph{Performance evaluation and parameter tuning} To evaluate performance and consensus, we use the stationarity measure studied in  \Cref{subsec:Stationarity and consensus bounds}, i.e.,
    $$
    \|\nabla f(\bar x^t)\|^2+\frac{\sL^2}{m}{\sum}_{i=1}^m\|\bar x^t-x_i^t\|^2.
    $$
This indicates how close the agents' iterates are to a common stationary point.
Similar measures can be found in \cite{sun2019distributed,hajinezhad2019zone}.
Each method selects the best-performing step size $\eta$ from the logarithmic grid $\{10^{0},10^{\pm1},\ldots,10^{\pm4}\}$, and {\sf DMGT} tunes the clipping parameter $\delta$ over the same grid.

\subsection{Phase retrieval}
\begin{figure}[t]
    \centering
    \includegraphics[scale=0.24]{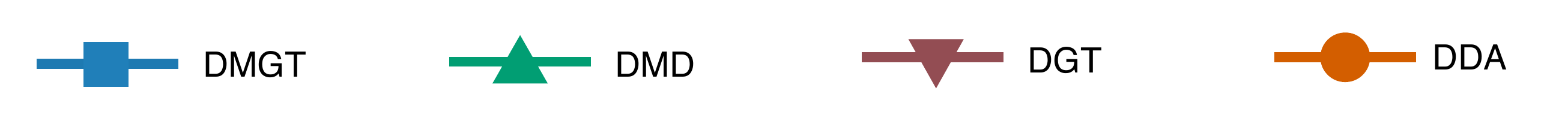}
\vspace{2mm}

\includegraphics[width=\linewidth]{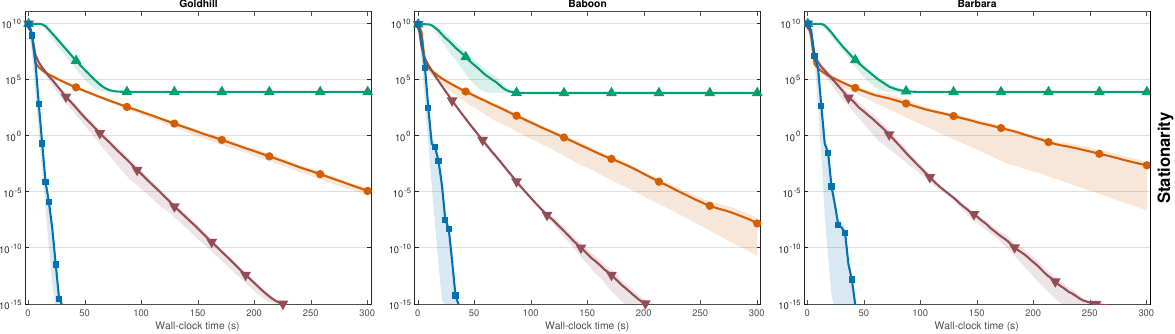}
\vspace{-.6cm}
    \caption{Performance of the decentralized methods on the phase retrieval problem. The curves show the mean, and the shaded regions show the range.}
    \label{fig:exp-1}
\end{figure}
We consider the phase retrieval problem, which is the distributed quadratic inverse problem studied in \cite{zhao2018distributed}. The problem is formulated as 
\begin{equation*}
\label{eq:quardratic inverse}
\mathop{\rm minimize}_{{x}_1,\cdots,\, {x}_m \in\mathbb{R}^d}\,\, \frac{1}{m}{\sum}_{i=1}^{m}\,\frac{1}{n}{\sum}_{\ell=1}^{n} \big(b_{i,\ell}-\langle a_{i,\ell},x_i\rangle^2\big)^2,  \qquad \text{s.t.} \quad x_i = x_j, \forall\, i,j\in[m],
\end{equation*}
where $a_{i,\ell}\in\mathbb{R}^d$ are sampling vectors and $b_{i,\ell}=\langle a_{i,\ell}, {x}_{\rm true}\rangle^2+\epsilon_{i,\ell}$ are measurement data  with ${x}_{\rm true}$ being the signal vector and $\epsilon_{i,\ell}$ being i.i.d. Gaussian noise.

\paragraph{Data generation} We consider three images: \texttt{Goldhill}, \texttt{Baboon} and \texttt{Barbara}. 
For each image, we reshape it into a vector and normalize it to obtain  $x_{\mathrm{true}}\in[0,1]^{d}$ with $d=1024$.
We use $m=32$ agents, where each agent $i\in[m]$ has a local dataset consisting of $n=200$ sensing vectors and the associated measurements.
For each $i\in[m]$ and $\ell\in[n]$, the sensing vector is sampled as $a_{i,\ell}\sim \mathcal{N}(0,I)$ and the additive noise is generated by $\epsilon_{i,\ell}\sim \mathcal{N}(0,0.1)$. 
To assess variability from data generation, we repeat each experiment with three random seeds. For each seed, we resample the sensing vectors $\{a_{i,\ell}\}$ and noise $\{\epsilon_{i,\ell}\}$, then generate the corresponding measurements $\{b_{i,\ell}\}$. The network and initialization are fixed across random seeds.

\paragraph{Kernel selection and initialization} We adopt the quartic kernel $h(x):=\frac{1}{4}\|x\|^4 + \frac12\|x\|^2$ for {\sf DMD}, {\sf DGT} and {\sf DMGT}. For {\sf DDA}, we consider a shifted kernel $\tilde h(x):=h(x-x^0)$, which is tailored to dual averaging-type updates to ensure that the initial point $x^0$ is the minimizer of the kernel \cite{liu2023rate}.  
By \cite[Lemma~5.1]{bolte2018first}, 
the phase retrieval problem is smooth 
relative to both $h$ and $\tilde h$.  All methods are initialized with a \emph{random positive Gaussian} vector $x^{0}={\rm abs}({\hat x^{0}})$, where $\hat x^{0}\sim{\cal N}(0,I)$.

\paragraph{Performance comparison}
As shown in \Cref{fig:exp-1}, {\sf DMGT} converges faster than the other methods across all three images.
{\sf DMD} initially reduces the stationarity measure but then stagnates because its consensus term remains high.
{\sf DGT} yields an improvement over {\sf DMD} by incorporating gradient tracking.
{\sf DDA} also decreases the stationarity measure steadily, but at a slower rate than {\sf DGT} and {\sf DMGT}.

In the reported runs, clipping occurs only during the early iterations of {\sf DMGT} and remains inactive thereafter. Across the three images and random seeds, the clipped steps account for at most 3\% of the total steps.

\begin{figure}[t]
    \centering
    \includegraphics[scale=0.24]{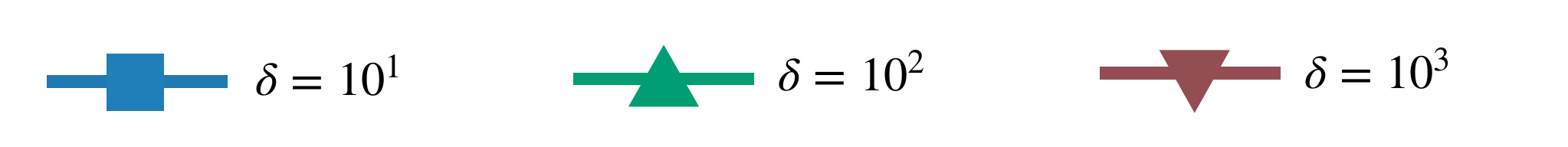}
\vspace{2mm}

\includegraphics[width=.98\linewidth]{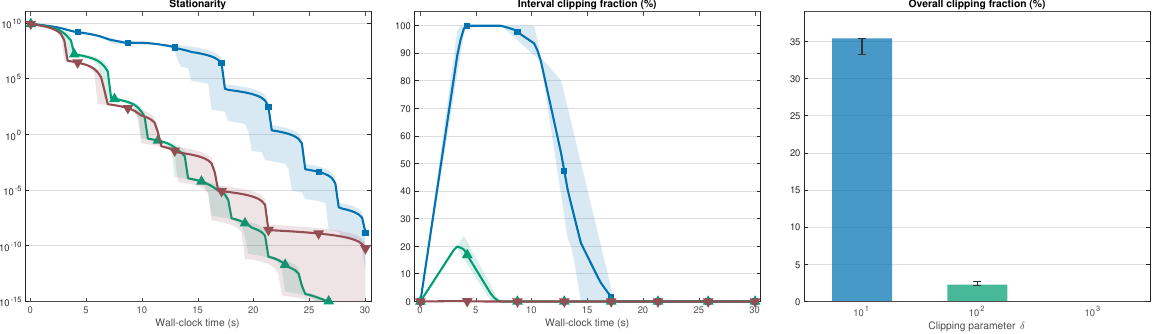}
    \caption{Sensitivity of {\sf DMGT} to the clipping parameter $\delta$ on \texttt{Goldhill}.}
    \label{fig:exp-1.5}
\end{figure}

\paragraph{Sensitivity to the clipping parameter} To assess the effect of $\delta$, we vary it for {\sf DMGT} on \texttt{Goldhill} and report the stationarity measure and the fraction of clipped steps against wall-clock time. As shown in \Cref{fig:exp-1.5}, larger values of $\delta$ lead to fewer clipped steps and generally yield lower stationarity values.

\subsection{Poisson inverse}\label{subsec:poisson inverse}
We evaluate the algorithms on Poisson inverse problems. We first conduct experiments on synthetic data, then assess performance on real image reconstruction tasks.
\subsubsection{Synthetic data}
Consider the Poisson inverse problem:
\[\mathop{\rm minimize}_{x_1,\cdots,x_m\in\mathbb{R}^d_{+}} \;\frac{1}{m}\,{\sum}_{i=1}^m {\cal D}_{\mathrm{KL}}(b_i,A_i x_i) \quad \text{s.t.}\quad x_i=x_j,\;\forall\, i,j\in[m].\]
Here, ${\cal D}_{\mathrm{KL}}(\cdot,\cdot)$ denotes the generalized Kullback--Leibler divergence \cite{csiszar1991least}. 

 \begin{figure}[t]
    \centering
\includegraphics[scale=0.24]{plots/legend.pdf}
\vspace{2mm}

\includegraphics[width=\linewidth]{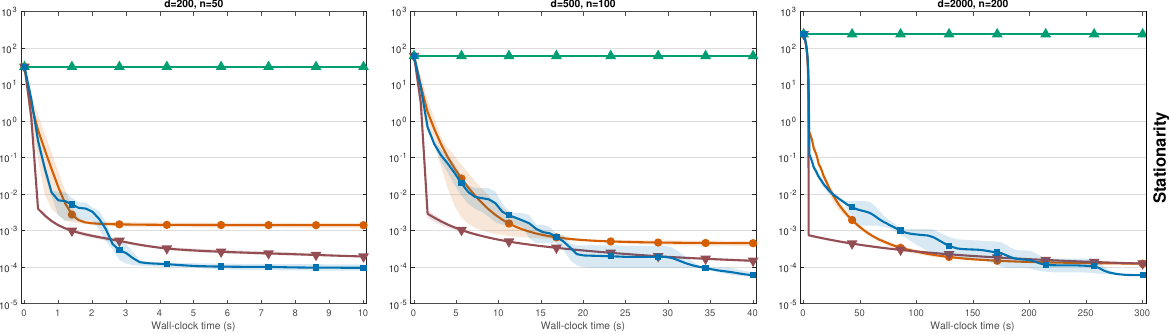}
\vspace{-.6cm}
    \caption{Performance of the decentralized methods on the synthetic Poisson inverse problem. The curves show the mean, and the shaded regions show the range. 
}
    \label{fig:exp-2}
\end{figure}
\paragraph{Data generation}
The ground-truth vector $x_{\mathrm{true}}\in\Rd_{+}$ is sampled uniformly from $[0,1]^d$.
We use $m=32$ agents.
On each agent $i\in[m]$, the local measurement matrix $A_i\in\mathbb{R}^{n\times d}_{+}$ is generated by sampling i.i.d.\ entries from a Student-$t$ distribution with $5$ degrees of freedom and taking absolute values.
Local observations $b_i\in\mathbb{R}^{n}_+$ are drawn according to the Poisson model $
b_i \sim \mathrm{Poisson}(A_i \, x_{\mathrm{true}})$ 
with the Poisson distribution applied componentwise. We test three scales: $(d,n)=(200,50)$, $(d,n)=(500,100)$, $(d,n)=(2000,200)$. 
We repeat each experiment with three random seeds, resampling the measurement matrices $\{A_i\}$ and corresponding observations $\{b_i\}$. The ground-truth vector, network, and initialization remain fixed.
\begin{figure}[t]
    \centering
\hspace{1cm}\includegraphics[scale=0.24]{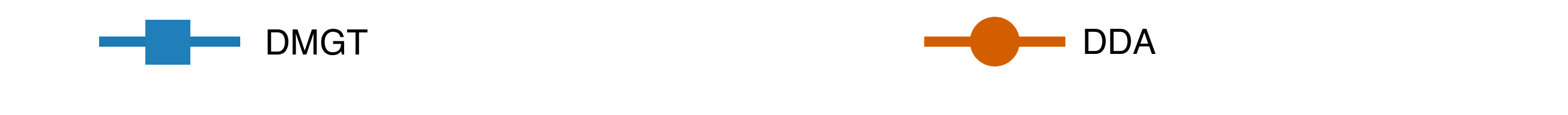}
\vspace{2mm}
\includegraphics[width=\linewidth]{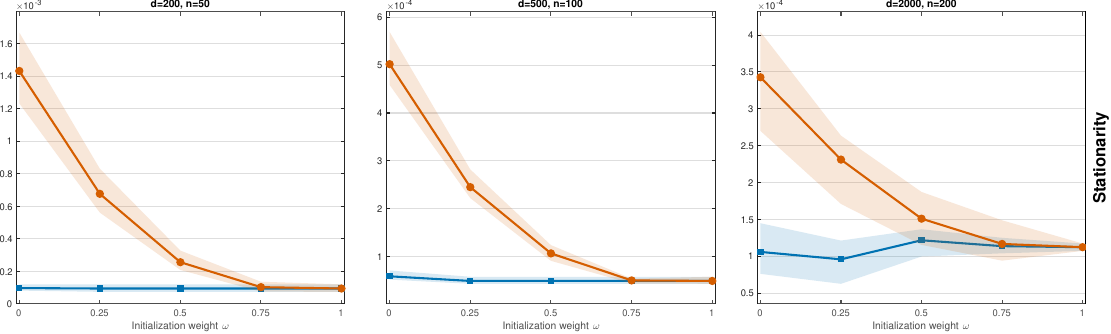}
\vspace{-.6cm}
    \caption{Best stationarity values attained under $x^0(\omega)$ within wall-clock budgets of $10$, $40$, and $300$ seconds for the three problem scales, respectively. The curves show the mean, and the shaded regions show the range.}
    \label{fig:initialization sensitivity}
\end{figure}
\paragraph{Kernel selection and initialization} For {\sf DMD}, {\sf DGT} and {\sf DMGT}, we adopt the regularized Burg's entropy $
h(x):=-\sum_{i=1}^{d}\log(x_i)+\frac{1}{2}\|x\|^{2}.$ Note that Poisson inverse problems are smooth relative to $h$, see \cite{bauschke2017descent}. 
For {\sf DDA}, we use the shifted kernel $
\tilde h(x):=h(x-x^{0}+\one)$, as suggested in \cite{liu2023rate}, to ensure the prescribed initial point $x^0\in\Rd_{++}$ is the minimizer of the kernel. 
All methods are initialized with a \emph{random positive Gaussian} vector $
x^{0}={\rm abs}({\hat x^{0}})$, $\hat x^{0}\sim{\cal N}(0,I)$, which is
consistent with the phase retrieval experiments. 
\paragraph{Performance comparison} As shown in \Cref{fig:exp-2}, {\sf DMGT} attains the lowest stationarity levels across all three problem scales, with fewer than 0.1\% clipped steps across all tested problems and random seeds. {\sf DMD} stagnates because its consensus error remains high, and gradient tracking enables {\sf DGT} to reduce this error and improve stationarity. {\sf DDA} makes initial progress but soon stagnates. Since its shifted kernel depends on the initial point, we next examine whether this stagnation is sensitive to initialization. 

\paragraph{Sensitivity to initialization}
Specifically, we compare {\sf DDA} and {\sf DMGT} under the weighted initializations:
$$ 
x^0(\omega):=(1-\omega){\rm abs}(\hat x^0)+\omega\one\qquad \text{with}\quad  \hat x^0\sim\mathcal N(0,I) \quad \text{and} \quad \omega\in[0,1].
$$
Here, $\omega=0$ yields the random initialization used in \Cref{fig:exp-2}. When $\omega=1$, we have $x^0=\one$, and the shifted kernel of {\sf DDA} reduces to the regularized Burg's entropy.

\Cref{fig:initialization sensitivity} reports the best stationarity value attained by each algorithm within the wall-clock budget. The result for {\sf DMGT} changes little with $\omega$, but {\sf DDA} is sensitive to the initialization.

\subsubsection{Image reconstruction}
Next, we apply these methods to reconstruct images from blurred and Poisson-corrupted observations.

\paragraph{Problem setup} For implementation, we crop all images to square matrices.
Let $X\in\mathbb{R}^{d\times d}_{+}$ denote the image, and let $\{{\cal A}_i\}_{i=1}^{m}$ be linear imaging operators, where each map ${\cal A}_i:\mathbb{R}^{d\times d}_+\to\mathbb{R}^{d\times d}_{+}$ models a blur operation\footnote{Implemented via \texttt{imfilter} in MATLAB.}.
Each agent $i\in[m]$ observes a blurred and Poisson-corrupted image $B_i\in\mathbb{R}^{d\times d}_{+}$ generated componentwise by $
B_i \sim \,\mathrm{Poisson}(\alpha\cdot {\cal A}_i(X_{\mathrm{true}}))/\alpha,$ 
where $X_{\mathrm{true}}\in[0,255]^{d\times d}$ is the ground truth and $\alpha>0$ controls the noise intensity.
Throughout, we set $\alpha=10$. In this experiment, we consider images with different resolutions: \texttt{Cameraman} ($d=256$) and \texttt{Peppers} ($d=512$).
The ground-truth images and the corrupted observations are shown in \Cref{fig:exp-3}.
\begin{figure}[t]
    \centering
\includegraphics[width=\linewidth]{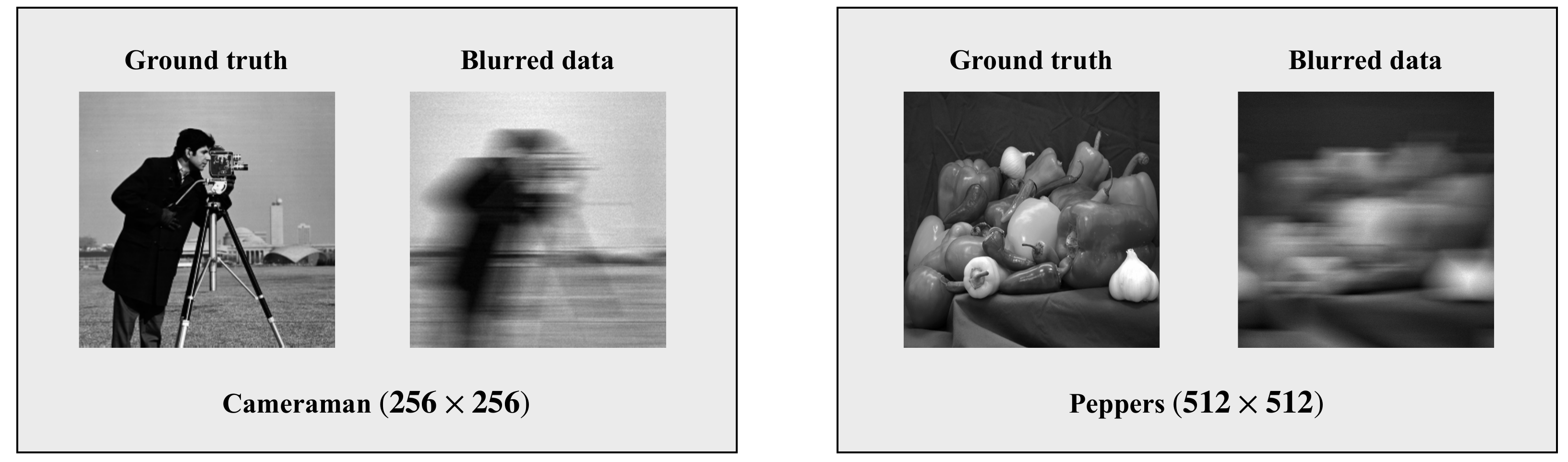}
\vspace{-4mm}
    \caption{The blurred image is the observation (from one of the agents) generated using motion blur with length $50$ for \texttt{Cameraman} and $100$ for \texttt{Peppers}, and corrupted by Poisson noise.}
    \label{fig:exp-3}
\end{figure}

To reconstruct the image from blurred and corrupted observations, we solve the Poisson inverse problem with a smoothed total variation (TV) penalty:
\[
\mathop{\rm minimize}_{X_1,\cdots,X_m\in\mathbb{R}^{d\times d}_{+}} \;
\frac{1}{m}{\sum}_{i=1}^{m} \Big[{\cal D}_{\mathrm{KL}}\!\big(B_i, {\cal A}_i(X_i)\big)
\;+\;
\lambda\cdot\mathrm{TV}(X_i)\Big]\quad \text{s.t.}\quad X_i=X_j,\;\forall\,i,j\in[m],
\]
where $\lambda=10^{-4}$ and the number of agents $m=8$.
The smoothed TV regularizer is defined as
\[\mathrm{TV}(X)
:= {\sum}_{p=1}^{d}{\sum}_{q=1}^{d}
\sqrt{\big[X_{(p+1,q)}-X_{(p,q)}\big]^2+\big[X_{(p,q+1)}-X_{(p,q)}\big]^2+\varepsilon_{\mathrm{TV}}^2},\]
with $\varepsilon_{\mathrm{TV}}=10^{-10}$ and the boundary convention $X_{(d+1,q)}=X_{(d,q)}$, $X_{(p,d+1)}=X_{(p,d)}$.

\paragraph{Kernel selection}
 For {\sf DMD}, {\sf DGT} and {\sf DMGT}, we employ the regularized Burg entropy
$h(X):=-\sum_{p=1}^{d}\sum_{q=1}^{d}\log\!\big(X_{(p,q)}\big)+\frac{1}{2}\|X\|_{F}^{2}$.
For {\sf DDA}, we use the shifted kernel $\tilde h(X):=h(X-X^{0}+\one)$, where $\one$ denotes the all-ones matrix.

\begin{figure}[t]
    \centering
\includegraphics[width=\linewidth]{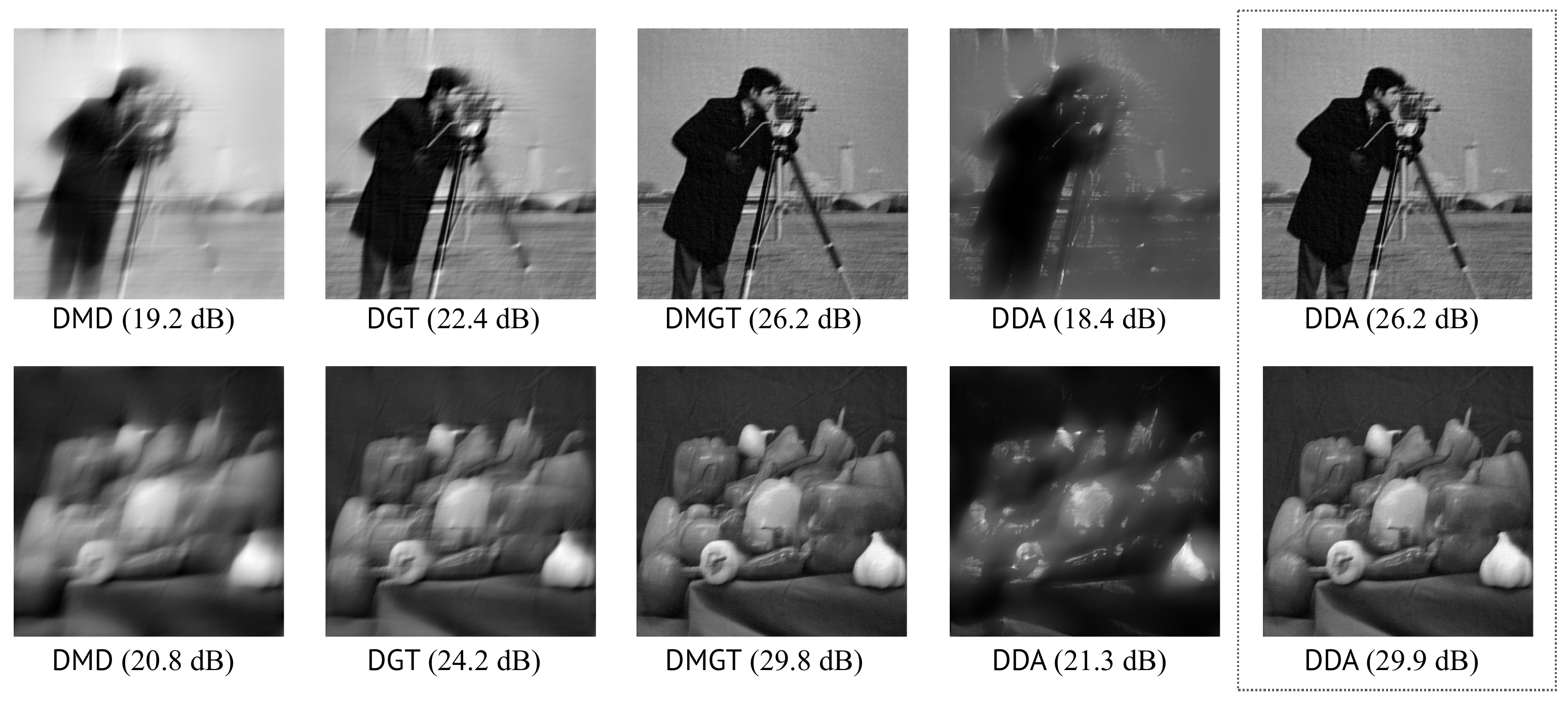}
\vspace{-6mm}
    \caption{Recovered images (with PSNR in dB) after $60$ seconds of wall-clock time for \texttt{Cameraman} (top row) and $240$ seconds for \texttt{Peppers} (bottom row).
The first four columns report reconstructions produced by {\sf DMD}, {\sf DGT}, {\sf DMGT} and {\sf DDA} under the observation-based initialization.
For {\sf DDA}, we additionally display the reconstruction obtained from a favorable initialization (dotted frame, last column).
}
    \label{fig:exp-4}
\end{figure}
\paragraph{Initialization} In imaging, a standard choice is the observation-based initialization $X^{0}_{\rm obs}:=\frac{1}{m}\sum_{i=1}^{m} B_i$.
Since {\sf DDA} is sensitive to initialization, we report its results under both $X^{0}_{\rm obs}$ and a favorable initialization $X^{0}_{\rm fav}$. Following the initialization-sensitivity experiment (see \Cref{fig:initialization sensitivity}), $X^{0}_{\rm fav}$ is selected by tuning the weight in a convex combination of $X^{0}_{\rm obs}$ and the all-ones matrix.

\paragraph{Performance comparison} We report reconstructions from blurred and Poisson-corrupted observations, together with the peak signal-to-noise ratio (PSNR), after $60$ and $240$ seconds of wall-clock time for \texttt{Cameraman} and \texttt{Peppers}, respectively\footnote{Since \texttt{Peppers} contains four times as many pixels as \texttt{Cameraman}, we use four times the wall-clock budget.}. 
As shown in \Cref{fig:exp-4}, {\sf DMGT} produces sharp reconstructions for both images. {\sf DMD} produces lower-quality reconstructions, and {\sf DGT} improves on these results through gradient tracking. With the observation-based initialization, {\sf DDA} produces poor reconstructions (fourth column of \Cref{fig:exp-4}). With the favorable initialization (dotted frame), its performance is comparable to {\sf DMGT}. The consensus error of {\sf DMD} remains above $10$, and all other methods attain consensus errors on the order of $10^{-5}$.  Clipping of {\sf DMGT} remains inactive throughout the reconstructions of both images.


%
%
%
\appendix
\section{Proof of \texorpdfstring{\Cref{lem:contraction}}{Lemma \getrefnumber{lem:contraction}}}
\label{appendix: Lemma-contraction} 
\begin{proof}
    Let $J=\frac1m\,\one\one^\top$, then direct computation gives $(I-J)W=(W-J)(I-J)$. Together with the update $\bv^+=W\bv + \bu$, we obtain
    \begin{equation*} \begin{aligned}
        &\hspace{5mm}(I-J)\bv^+  [H^+]^{\frac12} = (I-J)W\bv  [H^+]^{\frac12} + (I-J)\bu [H^+]^{\frac12} \\
        &=(W-J) (I-J)\bv H^{\frac12} H^{-\frac12}[H^+]^{\frac12} + (I-J)  \bu H^{\frac12} H^{-\frac12}[H^+]^{\frac12}.
    \end{aligned} \end{equation*}
    From the inequality $\|ABC\|_F \leq \|A\|\|B\|_F\|C\|$ for all $A,B,C$ of compatible dimensions, we deduce
\begin{equation}\label{eq:Appendix contraction}
\begin{aligned}
    \big\|(I\!-\!J)\bv^+ [H^+]^{\frac12}\!\big\|_{\!F}  &\leq \big\| H^{-\frac12}[H^+]^{\frac12} \big\| \Big( \|W\!-\!J\| \|(I\!-\!J)\bv  H^{\frac12}\|_F + \|I\!-\!J\| \| \bu H^{\frac12}\|_F \Big)\\
    & \leq  \sqrt{1+\alpha} \Big(\rho \| (I-J)\bv H^{\frac12}\|_F + \| \bu H^{\frac12}\|_F \Big),
   \end{aligned}
\end{equation}
where the last line is by $\big\| H^{-\frac12}[H^+]^{\frac12} \big\| = \big\| [H^+]^{\frac12} H^{-\frac12} \big\| \leq \sqrt{1+\alpha}$, $\|(I-J) \|=1$ and $\|W-J\|=\rho$.
Squaring both sides and invoking the inequality $(a+b)^2 \leq (1+\varepsilon)a^2 + (1+\varepsilon^{-1})b^2$ with $\varepsilon=(1-\rho)/(1+\rho)$, we obtain 
\begin{equation*} \begin{aligned}
    \big\|(I-J)\bv^+ [H^+]^{\frac12}\big\|_F^2 & \leq  (1+\alpha)\left(\frac{2\rho^2}{1+\rho} \, \big\| (I-J)\bv H^{\frac12} \big\|_F^2 + \frac{2}{1-\rho}\, \big\| \bu H^{\frac12} \big\|_F^2 \right)\\
    & \leq  \rho \big\| (I-J)\bv H^{\frac12} \big\|_F^2 + \frac{3}{1-\rho}  \big\| \bu H^{\frac12} \big\|_F^2,
\end{aligned} \end{equation*} 
where the last inequality is because $(1+\alpha)\cdot \frac{2\rho}{1+\rho} \leq 1$ and $1+\alpha \leq \frac32$. 
Upon substituting 
$\| \bu H^{\frac12}\|_F^2={\sum}_{i=1}^m \|u_i\|^2_{H}$, $\| (I-J)\bv H^{\frac12}\|_F^2 = {\sum}_{i=1}^m\|v_i - \bar v\|^2_{H}$ and the equality 
\[ \big\|(I-J)\bv^+ [H^+]^{\frac12} \big\|_F^2 = {\sum}_{i=1}^m \big\|v_i^+ - \bar v^+ \big\|^2_{H^+},\] 
we obtain the desired bound.
\end{proof}
 
\section{Complete proof of \texorpdfstring{\Cref{theorem:HURC-List}}{Theorem \getrefnumber{theorem:HURC-List}}}\label{sec:kernel verifications}
\subsection{Proof of \texorpdfstring{\Cref{lem:verify kernels}}{Proposition \getrefnumber{lem:verify kernels}}}\label{appendix:separable kernel}
\begin{proof}
Since $h$ is separable, we have $\cR_h(x,y) = \max_{i\in[d]}\;\cR_{h_i}(x,y)$. For each $i\in[d]$,
\begin{equation*} \begin{aligned}
\cR_{h_i}(x,y) = \bigg|\frac{h_i^{\prime\prime}(x_{(i)})}{h_i^{\prime\prime}(y_{(i)})}-1 \bigg|
   &\leq \exp\Big(\big|\log\big(h_i^{\prime\prime}(x_{(i)}) \big)-\log\big(h_i^{\prime\prime}(y_{(i)}) \big) \big|\Big)-1\\
   &= \exp\Big(\big|g_i(\phi_i(x_{(i)})) - g_i(\phi_i(y_{(i)}))\big| \Big) -1 .
\end{aligned} \end{equation*}
Note that $g_i$ is $\sG$-Lipschitz and $|\phi_i(x_{(i)})-\phi_i(y_{(i)})|\leq \sH  \delta$, it follows that 
\[\exp\Big(\big|g_i(\phi_i(x_{(i)})) - g_i(\phi_i(y_{(i)}))\big| \Big) \leq \exp(\sG\sH \delta),\qquad \forall\, i\in[d].\]
Therefore, $\cR_h(x,y) = \max_{i\in[d]}\,\cR_{h_i}(x,y) \leq \exp(\sG\sH \delta)-1$.
\end{proof}

\subsection{Proof of \texorpdfstring{\Cref{prop:HRUC self-concordant}}{Proposition \getrefnumber{prop:HRUC self-concordant}}}\label{appendix:HRUC self-concordant}
\begin{proof}
Fix $x,y\in\mathbb{R}^d$ and set $v := x-y$. Define the line segment and the associated Hessian:
\begin{equation*}
    x(s) := y + s v,\; s\in[0,1]\qquad \text{and} \qquad  H(s) := \nabla^2 h\big(x(s)\big).
\end{equation*}
By strong convexity, $H(s)\succ 0$ for all $s\in[0,1]$. For any nonzero vector $u\in\mathbb{R}^d$, define
\[
    \phi_u(s) := \frac{1}{2}\log\big(u^\top H(s)\, u\big),\quad \text{then}\quad  \phi_u^\prime(s)
    = \frac{1}{2}\,\frac{u^\top \dot H(s) u}{u^\top H(s)u}
    = \frac{1}{2}\,\frac{\nabla^3 h(x(s))[u,u,v]}{u^\top H(s)u}.
\]
By the self-concordance bound \cite[Lemma 5.1.2]{Nesterov2018}, we have
\[
    \big|\nabla^3 h(x(s))[u,u,v]\big|
    \;\le\; 2M\,\|u\|_{H(s)}^2\,\|v\|_{H(s)}
    = 2M\,\big(u^\top H(s)u\big)\,\|v\|_{H(s)}.
\]
Hence $|\phi_u^\prime(s) |\;\le\; M\,\|v\|_{H(s)}.$ 
Integrating yields 
\[
    \big|\phi_u(1) - \phi_u(0)\big|
    \;\leq\;
    M \int_0^1 \|v\|_{H(s)}\,{\rm d}s \quad \Longrightarrow \quad \left|
        \log\frac{u^\top H_1 u}{u^\top H_0 u}
    \right|
    \;\leq\;
    2M \int_0^1 \|v\|_{H(s)}\,{\rm d}s,
\]
where we denote $H_0:=H(0)$ and $H_1:=H(1)$ for brevity.
Thus, for all $u\neq 0$,
\[
    \exp\Big(-2M \int_0^1 \|v\|_{H(s)}\,{\rm d}s\Big)
    \;\leq\;
    \frac{u^\top H_1 u}{u^\top H_0 u}
    \;\leq\;
    \exp\Big(2M \int_0^1 \|v\|_{H(s)}\,{\rm d}s\Big).
\]
By the Rayleigh-quotient characterization, each eigenvalue $\lambda$ of $H_0^{-1/2}H_1H_0^{-1/2}$ lies in
\[
    \lambda\in\bigg[\exp\Big(-2M \int_0^1 \|v\|_{H(s)}\,{\rm d}s\Big),\exp\Big(2M \int_0^1 \|v\|_{H(s)}\,{\rm d}s\Big)\bigg].
\]
This, together with $H_0 = \nabla^2 h(y)$ and $H_1= \nabla^2 h(x)$, gives 
\begin{equation}\label{eq:verify self 1}
    \cR_h(x,y) 
    \;\leq\;
    \exp\Big(2M \int_0^1 \|v\|_{H(s)}\,{\rm d}s\Big) - 1.
\end{equation}
Note that 
$\nabla h(x) - \nabla h(y)
    = \int_0^1 H(s) v\,{\rm d}s$, then
\[
    v^\top\big(\nabla h(x) - \nabla h(y)\big)
    = \int_0^1 v^\top H(s) v\,{\rm d}s
    = \int_0^1 \|v\|_{H(s)}^2\,{\rm d}s.
\]
On one hand, the condition $\rho_h(x,y)\leq \delta$ implies 
\[
    \int_0^1 \|v\|_{H(s)}^2\,{\rm d}s = v^\top\big(\nabla h(x) - \nabla h(y)\big)
    \;\le\;
    \rho_h(x,y)\,\|v\|
    \leq \delta\,\|v\|.
\]
On the other hand, $\mu$-strong convexity of $h$, together with $\rho_h(x,y)\leq \delta$, yields
$
\|v\| = \|x-y\| \leq \rho_h(x,y)/\mu \leq \delta/\mu$. 
Therefore, $ \int_0^1 \|v\|_{H(s)}^2\,{\rm d}s \leq \delta^2/\mu$. By Jensen's inequality,
\[
    \int_0^1 \|v\|_{H(s)}\,{\rm d}s
    \;\leq \Big(\int_0^1 \|v\|_{H(s)}^2\,{\rm d}s\Big)^{1/2} \leq \delta/\sqrt{\mu}.
\]
Combining this bound with \eqref{eq:verify self 1}, we obtain
\[
   \hspace{2cm}\cR_h(x,y) 
    \;\leq\;
    \exp\!\bigg(2M \int_0^1 \|v\|_{H(s)}\,{\rm d}s\bigg) - 1
    \;\leq\;
    \exp\!\bigg(\frac{2M}{\sqrt{\mu}}\,\delta\bigg) - 1.
\]
\end{proof}

\subsection{Derivation of HRUC modulus of power kernels}\label{appendix:power kernel}
\begin{proof}
The gradient and Hessian of $h:\Rd\to\R$ are
\[
\nabla h(x)= (\mu+\|x\|^r)\cdot x \quad \text{and} \quad \nabla^2 h(x)= \big(\mu+\|x\|^r\big) \cdot  I + r\|x\|^{r-2}xx^\top.
\]
Moreover, by $\nabla^2 h(y) \succeq (\mu+\|y\|^r) I$, we have
\[   \cR_h(x,y)  \leq \frac{1}{\mu+\|y\|^r}\cdot \big\|\nabla^2 h(x) - \nabla^2 h(y)\big\|,\quad \forall\,x,y\in\Rd.
\]
Introduce the curve $x(s):=(1-s)\cdot y + s \cdot x$ with $s\in[0,1]$, then
for every $r>0$, the map $s\mapsto\nabla^2h(x(s))$ is absolutely continuous. Thus,
\begin{equation}\label{eq:power kernel 2}
    \cR_h(x,y)  \leq \frac{1}{\mu+\|y\|^r} \int_0^1\Big\|\frac{{\rm d}\, \nabla^2 h(x(s))}{{\rm d}s}\Big\|\;{\rm d}s \overset{\text{(i)}}{\leq} \frac{1}{\mu+\|y\|^r} \int_0^1\big\|\nabla^3 h(x(s)) [x-y] \big\|\;{\rm d}s,
\end{equation}
where (i) is by chain rule ${\rm d} \nabla^2 h(x(s))/{\rm d}s = \nabla^3 h(x(s)) \,[\dot x(s)] = \nabla^3 h(x(s)) [x-y]$ for almost every $s\in[0,1]$
and $\nabla^3 h(x)[\cdot]$ is the directional derivative of $\nabla^2 h(x)$ at $x$. Next, consider
\begin{equation*}
\begin{aligned}
&\hspace{4.5mm}\big\|\nabla^3 h(x)[u] \big\| \\ &= \Big\| r\|x\|^{r-2}(x^\top u)I + r(r-2)\|x\|^{r-4}(x^\top u) xx^\top + r\|x\|^{r-2}(ux^\top + x u^\top)\Big\|    \\
&\leq \big(3r+|r(r-2)|\big)\|x\|^{r-1} \|u\|.
\end{aligned} \end{equation*}
Using this estimate, we can further write \eqref{eq:power kernel 2} as
\begin{equation}\label{eq:power kernel 3}
    \cR_h(x,y)  \leq \frac{\big(3r+|r(r-2)|\big)\|x-y\|}{\mu+\|y\|^r} \cdot \int_{0}^1 \|x(s)\|^{r-1} {\rm d}s.
\end{equation}
\noindent 
\textbf{Case I: $r> 1$.} Indeed, $\rho_h(x,y)\leq \delta$ implies $\|x-y\| \leq \delta/\mu$ and $\|x(s)\| = \|y+s(x-y)\| \leq \|y\|+\delta/\mu$. If $y=0$, then $\int_0^1\|x(s)\|^{r-1}{\rm d}s=\|x\|^{r-1}/r$, and \eqref{eq:power kernel 3} yields the claimed bound. Below we assume $\|y\|>0$. Then, \eqref{eq:power kernel 3} becomes
\begin{equation}\label{eq:power kernel 4}
\cR_h(x,y)  \leq \frac{(r^2+3r)\delta}{\mu}\cdot \frac{(\|y\| + \delta/\mu)^{r-1}}{\|y\|^r+\mu} 
\leq \frac{4\delta r^2 }{\mu}\cdot 2^{r-1} \bigg(\frac{1}{\|y\|+\mu/\|y\|^{r-1}} + \frac{\delta^{r-1}}{\mu^r}\bigg). 
\end{equation}
By Young's inequality $a^p/p+b^q/q \geq ab$ for all $a,b\geq 0,p,q\geq 1$ and $1/p+1/q=1$, set $p=r/(r-1)$, $q=r$, it holds for all $t>0$ that 
\[
t+ \mu/t^{r-1} \geq  \frac{\big(t^{1/p}\big)^p}{p} + \frac{\big(\mu^{1/q}/t^{(r-1)/q}\big)^q}{q} \geq \mu^{1/r}.
\] 
Substituting $t=\|y\|$ in the above estimate,
and merging the obtained bound into \eqref{eq:power kernel 4} leads to $$\cR_h(x,y)  \leq \frac{2^{r+1}r^2}{\mu}\Big(\frac{\delta}{\mu^{1/r}}+\frac{\delta^r}{\mu^r}\Big).$$


\noindent 
\textbf{Case II: $0< r\leq  1$.} To upper bound $\int_{0}^1 \|x(s)\|^{r-1} \,{\rm d}s$, we introduce
\[
z:=x-y\quad \text{and} \quad v:=\frac{x-y}{\|x-y\|}.
\]
By Schwarz's inequality, then $\|x(s)\| = \|y + sz\|\cdot \|v\| \geq | \langle y+sz,v\rangle  | = | \langle y,v\rangle + s\|z\| |$.
With this bound, the integral is further written as
\begin{equation*}
\int_{0}^1 \|x(s)\|^{r-1} {\rm d}s \leq \int_0^1 \big| \langle y,v\rangle + s\|z\| \big|^{r-1} {\rm d}s = \|z\|^{r-1} \cdot \int_{\beta}^{ \beta +1} |t|^{r-1} {\rm d}t\qquad \text{with}\quad  \beta:=\frac{\langle y,v\rangle}{\|z\|}    
\end{equation*}
after changing variable $t := s + \beta$. Since $\int_{\beta}^{ \beta +1} |t|^{r-1} {\rm d}t \leq 2\int_0^1 t^{r-1} {\rm d}t = \frac{2}{r},$ we obtain 
$
\int_{0}^1 \|x(s)\|^{r-1} {\rm d}s \leq \frac{2}{r}\|z\|^{r-1}$, 
 which, combined with \eqref{eq:power kernel 3}, leads to 
\[
\cR_h(x,y)
\leq \frac{[3r+r(2-r)]\|z\|}{\mu+\|y\|^r}\cdot\frac{2\|z\|^{r-1}}{r}
\leq \frac{10\|z\|^r}{\mu}
\leq \frac{10\delta^r}{\mu^{r+1}}.
\]
In view of \textbf{Case I \& II}, we obtain a uniform upper bound when $r>0$, i.e., \[
\cR_h(x,y) \leq \zeta(\delta):=\frac{\max\{10, r^2 2^{r+1}\}}{\mu}\Big(\frac{\delta}{\mu^{1/r}}+\frac{\delta^r}{\mu^r}\Big). \]
This completes the proof.
\end{proof}

\section{Preparatory tools}
\begin{lemma}[Three points identity \cite{chen1993convergence}] 
\label{lem:three point}
Let $h: \mathbb{R}^d \to (-\infty,+\infty]$ be a kernel function associated with $\mathcal{Z}$. For any
$x \in \dom{h}$, and $y,z \in \idom{h}$, we have
\begin{equation*}
\cD(x,z) - \cD(x,y) - \cD(y,z) = \langle\nabla h(y) - \nabla h(z), x-y\rangle.    
\end{equation*}
\end{lemma}

\begin{lemma}\label{lem:matrix norm bound}
Given symmetric matrices $A,B \succ 0$ and the constant $\alpha>0$. If $\|B^{-\frac12}AB^{-\frac12}-I\| \leq \alpha$, then $\|A^{\frac12}B^{-\frac12}\|\leq \sqrt{1+\alpha}$.
\end{lemma}
\begin{proof}
    Indeed, $\|A^{\frac12}B^{-\frac12} \|^2 = \|B^{-\frac12} A B^{-\frac12}\|$. Since $\|B^{-\frac12}AB^{-\frac12}-I\| \leq \alpha$, we have $\|B^{-\frac12}AB^{-\frac12}\|\leq 1+\alpha$. Thus, $\big\|A^{\frac12}B^{-\frac12} \big\|^2 \leq 1+\alpha$.
\end{proof}

\section{Proof of technical lemmas}
\subsection{Proof of \texorpdfstring{\Cref{lem:first descent}}{Lemma \getrefnumber{lem:first descent}}}\label{proof of lem:first descent}
\begin{proof}
By the update \eqref{eq:update bar z} of $\bar z^t=\nabla h(\bar x^t)$ and the three points identity (\Cref{lem:three point}), we get 
\[\langle \bar s^t, \bar x^{t+1} - \bar x^t \rangle = \langle \nabla h(\bar x^t) - \nabla h(\bar x^{t+1}) , \bar x^{t+1} - \bar x^t \rangle = - \cD(\bar x^{t+1},\bar x^t) - \cD(\bar x^{t},\bar x^{t+1}),\]
which, along with the extended descent \eqref{eq:extended descent}, gives
    \begin{equation}
        \label{eq:descent 1}
          \begin{aligned}
        f(\bar x^{t+1}) & \leq f(\bar x^t) + \langle \nabla f(\bar x^t) , \bar x^{t+1} - \bar x^t \rangle + \sL \cdot \cD(\bar x^{t+1},\bar x^t)\\
        &\leq f(\bar x^t) + \langle \nabla f(\bar x^t) - \sL \bar s^t, \bar x^{t+1} - \bar x^t \rangle - \sL \cdot \cD(\bar x^t,\bar x^{t+1}).
    \end{aligned}
    \end{equation}
    For any vector $v\in\Rd$, note that $\bar x^t=\nabla h^*(\bar z^t)$ and $\bar z^{t+1} = \bar z^t - \bar s^t$, by mean value theorem, there exists a constant $\theta_t\in[0,1]$ such that
    \begin{equation}\label{eq:mean value of bar x} 
       \langle v, \bar x^{t+1} - \bar x^t \rangle = \langle v,\nabla h^*(\bar z^{t+1}) - \nabla h^*(\bar z^{t}) \rangle  = -v^\top\big[\nabla^2 h^*(\bar z^t - \theta_t \bar s^t)\big]\,\bar s^t. 
    \end{equation}  
    Substituting $v=\nabla f(\bar x^t) - \sL \bar s^t$ into \eqref{eq:mean value of bar x} and
   denoting $\bar z_\theta^t:= \bar z^t - \theta_t \bar s^t$ and $ H_{\theta_t}^*:=\nabla^2 h^*(\bar z_\theta^t)$, then we can express the inner product term in \eqref{eq:descent 1} as
   \begin{equation}
        \label{eq:inner product 1}
   \begin{aligned}
     \langle \nabla f(\bar x^t) - \sL \bar s^t, \bar x^{t+1} - \bar x^t \rangle &= 
       -\langle \nabla f(\bar x^t), \bar s^t \rangle_{H_{\theta_t}^*} +  \sL \|\bar s^t\|^2_{H_{\theta_t}^*}\\
       &\leq -\big \langle \nabla f(\bar x^t), \tfrac1m {\textstyle \sum}_{i=1}^m \, s_i^t \big \rangle_{H_{\theta_t}^*} + \frac{\sL}{m}{\sum}_{i=1}^m \|s_i^t\|^2_{H_{\theta_t}^*}.
   \end{aligned}
   \end{equation}
   The rest of proof is to bound the following inner product:
   \begin{equation}
        \label{eq:inner product 2}
        \begin{aligned}
           & \hspace{5mm} -\big \langle \nabla f(\bar x^t), \tfrac1m {\textstyle \sum}_{i=1}^m \, s_i^t \big \rangle_{H_{\theta_t}^*} \\
           &= - \frac1m {\sum}_{i=1}^m 
            \big\langle (\nabla f(\bar x^t) - \bar y^t) + (\bar y^t - y_i^t) + y_i^t , s_i^t \big\rangle_{H_{\theta_t}^*} \\
            &= \underbracket[.3mm][2mm]{- \frac1m {\sum}_{i=1}^m \langle y_i^t , s_i^t\rangle_{H_{\theta_t}^*}}_{T_1} + \underbracket[.3mm][2mm]{\frac1m {\sum}_{i=1}^m \langle y_i^t - \bar y^t, s_i^t \rangle_{H_{\theta_t}^*}}_{T_2} + \underbracket[.3mm][2mm]{\langle \bar y^t - \nabla f(\bar x^t), \bar s^t \rangle_{H_{\theta_t}^*}}_{T_3}.
        \end{aligned}
    \end{equation}
    First, we bound $T_1$. Since $s_i^t=\clip(y_i^t)$, we have $y_i^t=s_i^t\cdot \max\{1/\eta,\|y_i^t\|/\delta\}$. Thus, 
    \begin{equation}
        \label{eq:bound T1}
        \langle y_i^t, s_i^t \rangle_{H_{\theta_t}^*} \geq \frac1\eta \cdot \|s_i^t\|^2_{H_{\theta_t}^*} \qquad \Longrightarrow \qquad T_1 \leq -\frac{1}{m\eta} \cdot {\sum}_{i=1}^m\|s_i^t\|^2_{H_{\theta_t}^*}.
    \end{equation}
We directly bound $T_2$ via Cauchy-Schwarz inequality:
    \begin{equation}
        \label{eq:bound T2}
        T_2 \leq \frac{1}{2m \alpha_1} {\sum}_{i=1}^m\|y_i^t - \bar y^t\|^2_{H_{\theta_t}^*} + \frac{\alpha_1}{2m} {\sum}_{i=1}^m\|s_i^t\|^2_{H_{\theta_t}^*}.
    \end{equation}
    To bound $T_3$, we invoke the update $\by^{t+1} = W \by^t + \nabla \bff(\bx^{t+1}) - \nabla \bff(\bx^{t})$, together with the double stochasticity of $W$, the average $\bar y^t = \frac{1}{m} (\by^t)^\top \one$, and the fact that $[\nabla \bff(\bx^t)]^\top \one = \sum_{i=1}^m \nabla f_i(x_i^t)$, this yields
    \begin{equation*} \begin{aligned}
        &\hspace{1.2cm}\bar y^{t+1} = \bar y^t + \frac{1}{m}{\sum}_{i=1}^m \nabla f_i(x_i^{t+1}) - \frac{1}{m}{\sum}_{i=1}^m \nabla f_i(x_i^{t}) \\[2mm]
        &\Longleftrightarrow \quad \bar y^t - \frac{1}{m}{\sum}_{i=1}^m \nabla f_i(x_i^{t}) = \cdots = \bar y^0 - \frac{1}{m}{\sum}_{i=1}^m \nabla f_i(x_i^{0}).
    \end{aligned} \end{equation*}
    Since $\bar y^0 = \frac{1}{m}{\sum}_{i=1}^m \nabla f_i(x_i^{0})$, we have $\bar
    y^t = \frac{1}{m}{\sum}_{i=1}^m \nabla f_i(x_i^{t})$. Hence, $T_3$ can be expressed as:
    \begin{equation*} \begin{aligned}
        T_3 &= \frac{1}{m} {\sum}_{i=1}^m \langle \nabla f_i(x_i^t) - \nabla f_i(\bar x^t), \bar s^t \rangle_{H_{\theta_t}^*}\\
        &\leq \frac{1}{2m\alpha_2} {\sum}_{i=1}^m \|\nabla f_i(x_i^t) - \nabla f_i(\bar x^t)\|_{H_{\theta_t}^*}^2 + \frac{\alpha_2}{2}\|\bar s^t\|_{H_{\theta_t}^*}^2\\
        &\leq \frac{1}{2m\alpha_2} {\sum}_{i=1}^m \|\nabla f_i(x_i^t) - \nabla f_i(\bar x^t)\|_{H_{\theta_t}^*}^2 + \frac{\alpha_2}{2m} {\sum}_{i=1}^m\|s^t_i\|_{H_{\theta_t}^*}^2.
    \end{aligned} \end{equation*}
    Applying \Cref{lem:dual Lipcshitz} and the bounds $\|\bar z^t-\bar z_{\theta_t}^t\|\leq\delta$ and $\|z_i^t-\bar z_{\theta_t}^t\|\leq\|z_i^t-\bar z^t\|+\|\bar z^t-\bar z_{\theta_t}^t\|\leq\frac{\sqrt{20m}}{1-\rho}\delta$ from \Cref{lem:iterates bounds}, we have 
    \begin{equation*} 
        \|\nabla f_i(x_i^t) - \nabla f_i(\bar x^t)\|_{H_{\theta_t}^*}^2 
        \leq \sL^2 \lambda^2 \cdot  \|\bar z^t - z_i^t\|_{H_{\theta_t}^*}^2\qquad 
    \forall\, i\in[m].
    \end{equation*} 
    Therefore, the term $T_3$ can be bounded by
    \begin{equation}
        \label{eq:bound T3}
    T_3 \leq \frac{\sL^2\lambda^2}{2m\alpha_2} {\sum}_{i=1}^m \|\bar z^t - z_i^t\|_{H_{\theta_t}^*}^2 + \frac{\alpha_2}{2m} {\sum}_{i=1}^m\|s^t_i\|_{H_{\theta_t}^*}^2.
    \end{equation}
   Combining the estimates \eqref{eq:inner product 1}--\eqref{eq:bound T3}, we obtain
   \begin{equation*} \begin{aligned}
        \langle \nabla f(\bar x^t) - \sL \bar s^t, \bar x^{t+1} - \bar x^t \rangle
       \leq &-\Big(\frac{1}{\eta}-\frac{2\sL+\alpha_1+\alpha_2}{2}\Big)\cdot \frac{1}{m}{\sum}_{i=1}^m\|s^t_i\|_{H_{\theta_t}^*}^2 \\&+ \frac{1}{2m \alpha_1} {\sum}_{i=1}^m\|y_i^t - \bar y^t\|^2_{H_{\theta_t}^*} + \frac{\sL^2\lambda^2}{2m\alpha_2} {\sum}_{i=1}^m \|\bar z^t - z_i^t\|_{H_{\theta_t}^*}^2.
   \end{aligned} \end{equation*}
Merging the above estimate into \eqref{eq:descent 1}, we thus complete the proof.
\end{proof}

\subsection{Proof of \texorpdfstring{\Cref{lem:consensus errors}}{Lemma \getrefnumber{lem:consensus errors}}}\label{proof of lem:consensus errors}
\begin{proof}
Given the bound \eqref{eqn:H-rel-diff}, applying \Cref{lem:contraction} with
    \[
    \bv=\bz^t,\quad  \bv^+=\bz^{t+1},\quad \bu=-W\bs^t,\quad H=H^*_{\theta_t} \quad  \text{and} \quad  H^+=H_{\theta_{t+1}}^*,
    \]
    we obtain the following bound 
    \[
    {\sum}_{i=1}^m \|\bar z^{t+1} - z_i^{t+1}\|_{H_{\theta_{t+1}}^*}^2 \leq \rho \cdot {\sum}_{i=1}^m \|\bar z^{t} - z_i^{t}\|_{H_{\theta_{t}}^*}^2 + \frac{3}{1-\rho} \cdot {\sum}_{i=1}^m\|(W\bs^t)_i\|_{H_{\theta_t}^*}^2.
    \]
Note that 
\begin{equation}
\label{eq:bound of Ws}
   {\sum}_{i=1}^m\|(W\bs^t)_i\|_{H_{\theta_t}^*}^2 = \Big\|W\bs^t\big[H^*_{\theta_t}\big]^{1/2}\Big\|_F^2 \leq \|W\|^2 \cdot \Big\|\bs^t\big[H^*_{\theta_t}\big]^{1/2}\Big\|_F^2 = {\sum}_{i=1}^m
\|s^t_i\|_{H_{\theta_t}^*}^2,
\end{equation}
where the inequality is by $\|AB\|_F \leq \|A\| \|B\|_F$ and the last equality is due to $\|W\|=1$ inferred from symmetric double stochasticity. 
Based on \eqref{eq:bound of Ws}, we obtain the bound
\begin{equation}\label{eq:consensus z}
         {\sum}_{i=1}^m \|\bar z^{t+1} - z_i^{t+1}\|_{H_{\theta_{t+1}}^*}^2 
 \leq \rho \cdot {\sum}_{i=1}^m \|\bar z^{t} - z_i^{t}\|_{H_{\theta_{t}}^*}^2 + \frac{3}{1-\rho} \cdot {\sum}_{i=1}^m
\|s^t_i\|_{H_{\theta_t}^*}^2.
\end{equation}  
Again, invoking \Cref{lem:contraction} with
    \[
    \bv=\by^t,\quad  \bv^+=\by^{t+1},\quad \bu=\nabla \bff(\bx^{t+1}) - \nabla \bff(\bx^{t}),\quad H=H^*_{\theta_t} \quad  \text{and} \quad  H^+=H_{\theta_{t+1}}^*,
    \]
    which, together with $\nabla \bff(\bx)=[(\nabla f_1(x_1))^\top,\cdots,(\nabla f_m(x_m))^\top]^\top$, yields
  \begin{equation}
        \label{eq:consensus bound for y 1}
        \begin{aligned}
       &\hspace{5mm} {\sum}_{i=1}^m\|y_i^{t+1} - \bar y^{t+1}\|^2_{H_{\theta_{t+1}}^*}\\ 
       &\leq \rho  {\sum}_{i=1}^m\|y_i^t - \bar y^t\|^2_{H_{\theta_t}^*} + \frac{3}{1-\rho}\cdot {\sum}_{i=1}^m \|\nabla f_i(x_i^{t+1})-\nabla f_i(x_i^t)\|_{H^*_{\theta_t}}^2\\
       &\leq \rho  {\sum}_{i=1}^m\|y_i^t - \bar y^t\|^2_{H_{\theta_t}^*} + \frac{3\sL^2\lambda^2}{1-\rho}\cdot {\sum}_{i=1}^m \|z_i^{t+1} - z_i^t\|_{H^*_{\theta_t}}^2,
    \end{aligned}
    \end{equation}
    where the last line is due to dual Lipschitz bound \Cref{lem:dual Lipcshitz} and the bound from \Cref{lem:iterates bounds}, i.e.,
    \[
    \|z_i^\tau-\bar z_{\theta_t}^t\|\leq\|z_i^\tau-\bar z^\tau\|+\|\bar z^\tau-\bar z_{\theta_t}^t\|\leq\frac{\sqrt{3m}}{1-\rho}\delta+\delta\leq\frac{\sqrt{20m}}{1-\rho}\delta,\qquad \tau\in\{t,t+1\}.
    \]
    From the update $\bz^{t+1}  = W (\bz^t -\bs^t)$, we deduce
    \begin{equation*} \begin{aligned}
        {\sum}_{i=1}^m \|z_i^{t+1} - z_i^t\|_{H^*_{\theta_t}}^2 &= \Big\|(\bz^{t+1} - \bz^t)\cdot  \big[H^*_{\theta_t}\big]^{1/2}\Big\|^2_F\\
        &=\Big\|\big[(W-I) \bz^t - W\bs^t \big]\cdot \big[H^*_{\theta_t}\big]^{1/2}\Big\|^2_F\\
        &\overset{\text{(i)}}{=}\Big\|\big[(W-I) (\bz^t - \one (\bar z^t)^\top) - W\bs^t \big] \cdot \big[H^*_{\theta_t}\big]^{1/2}\Big\|^2_F\\
        &\overset{\text{(ii)}}{\leq} \frac{4}{3}\, \Big\|(W-I) (\bz^t - \one (\bar z^t)^\top) \cdot \big[H^*_{\theta_t}\big]^{1/2} \Big\|^2_F + 4 \Big\|W\bs^t \big[H^*_{\theta_t}\big]^{1/2} \Big\|^2_F\\
        &\overset{\text{(iii)}}{\leq} \frac{16}{3} \, {\sum}_{i=1}^m \|\bar z^{t} - z_i^{t}\|_{H_{\theta_{t}}^*}^2 + 4\, {\sum}_{i=1}^m
    \|s^t_i\|_{H_{\theta_t}^*}^2,
    \end{aligned} \end{equation*}
    where (i) uses the fact that $(W-I) \cdot  \one(\bar z^t)^\top=0$, (ii) is by $\|A+B\|_F^2 \leq (1+\varepsilon)\|A\|_F^2 + (1+\varepsilon^{-1})\|B\|_F^2$ with $\varepsilon=1/3$, and (iii) invokes the estimate \eqref{eq:bound of Ws} and the relations $\|AB\|_F \leq \|A\|\|B\|_F$ and $\|W-I\|\leq 2$. Plugging the bound of $\sum_{i=1}^m \|z_i^{t+1} - z_i^t\|_{H^*_{\theta_t}}^2$ into \eqref{eq:consensus bound for y 1}, then
    \begin{equation}
        \label{eq:consensus bound for y 2}
        \begin{aligned}
    {\sum}_{i=1}^m\|y_i^{t+1} - \bar y^{t+1}\|^2_{H_{\theta_{t+1}}^*} \leq  &\;\rho \, {\sum}_{i=1}^m\|y_i^t - \bar y^t\|^2_{H_{\theta_t}^*} \\& + \frac{16\sL^2\lambda^2}{1-\rho} {\sum}_{i=1}^m \|\bar z^{t} - z_i^{t}\|_{H_{\theta_{t}}^*}^2 + \frac{12\sL^2\lambda^2}{1-\rho} {\sum}_{i=1}^m\|s^t_i\|_{H_{\theta_t}^*}^2.
        \end{aligned}
     \end{equation}
Multiplying \eqref{eq:consensus z} by $\xi=32\sL^2\lambda^2/(1-\rho)^2$ and summing it with \eqref{eq:consensus bound for y 2}, we get 
\begin{equation*} \begin{aligned}
     &{\sum}_{i=1}^m \Big(\|y_i^{t+1} - \bar y^{t+1}\|^2_{H_{\theta_{t+1}}^*} + \xi \|\bar z^{t+1} - z_i^{t+1}\|_{H_{\theta_{t+1}}^*}^2 \Big)\\[1mm] \leq & \frac{1+\rho}{2} \cdot {\sum}_{i=1}^m \Big(\|y_i^t - \bar y^t\|^2_{H_{\theta_t}^*} + \xi\|\bar z^{t} - z_i^{t}\|_{H_{\theta_{t}}^*}^2\Big) + \frac{108\sL^2\lambda^2}{(1-\rho)^{3}} \cdot  {\sum}_{i=1}^m\|s^t_i\|_{H_{\theta_t}^*}^2.
\end{aligned} \end{equation*}
By definition, $\cE_t=\frac1m {\sum}_{i=1}^m (\|y_i^t - \bar y^t\|^2_{H_{\theta_t}^*} + \xi\|\bar z^{t} - z_i^{t}\|_{H_{\theta_{t}}^*}^2)$. Therefore,
\begin{equation*}
      \cE_{t+1} \leq \frac{1+\rho}{2}\cdot \cE_t + \frac{108\sL^2\lambda^2}{(1-\rho)^{3}} \cdot \frac1m{\sum}_{i=1}^m\|s^t_i\|_{H_{\theta_t}^*}^2,
\end{equation*}
which completes the proof.
\end{proof}

\subsection{Proof of \texorpdfstring{\Cref{prop:descent}}{Lemma \getrefnumber{prop:descent}}}\label{appendix:proof of prop:descent}
\begin{proof}
    Setting $\alpha_1=16\sL/(1-\rho)$ and $\alpha_2=(1-\rho)\sL/2$ in \Cref{lem:first descent}, we have 
    \begin{equation}\label{eq:prop descent}
    f(\bar x^{t+1}) \leq  f(\bar x^t) - \sL \cD(\bar x^t,\bar x^{t+1}) -\bigg(\frac{1}{\eta}-\frac{9.25\,\sL}{1-\rho}\bigg)\cdot \frac{1}{m}{\sum}_{i=1}^m\|s^t_i\|_{H_{\theta_t}^*}^2 + \frac{1-\rho}{32\sL}\cdot \cE_t. 
    \end{equation}
    Next we invoke the consensus bound from \Cref{lem:consensus errors}. Multiplying \eqref{eq:consensus E} by $1/(8\sL)$, combining the result with \eqref{eq:prop descent}, and rearranging terms yield the sufficient descent:
     \begin{equation*} \begin{aligned}
       \cM_{t+1}&=f(\bar x^{t+1}) + \frac{1}{8\sL} \cdot \cE_{t+1}\\[1mm]
       &\leq 
       f(\bar x^t) - \bigg\{\frac{1}{\eta}-\bigg[\frac{9.25}{1-\rho} + \frac{13.5\lambda^2}{(1-\rho)^{3}}\bigg]\sL \bigg\}\cdot \frac{1}{m}{\sum}_{i=1}^m\|s^t_i\|_{H_{\theta_t}^*}^2 - \sL \cdot \cD(\bar x^t,\bar x^{t+1}) \\&\hspace{1.3cm}+ \Big(\frac{1}{8\sL} - \frac{1-\rho}{32 \sL} \Big) \cdot \cE_t\\[1mm]
       &\leq \cM_t - \bigg[\frac{1}{\eta} - \frac{22.75\cdot \lambda^2\sL}{(1-\rho)^{3}} \bigg]\cdot \frac{1}{m}{\sum}_{i=1}^m\|s^t_i\|_{H_{\theta_t}^*}^2 - \sL \cdot \cD(\bar x^t,\bar x^{t+1}) - \frac{1-\rho}{32\sL}\cdot \cE_t,
    \end{aligned} \end{equation*}
    where the last inequality holds by substituting the definition of $\cM_t$ and invoking $1/(1-\rho)\geq 1$ and $\lambda \geq 1$. Finally, we complete the proof by noting that 
    \begin{equation*}
        0<\eta \leq \frac{(1-\rho)^{3}}{25\lambda^2\sL} \qquad \Longrightarrow \qquad \frac1\eta - \frac{22.75\cdot \lambda^2\sL}{(1-\rho)^{3}} \geq \frac{1}{12\eta}. 
     \end{equation*}
\end{proof}

\subsection{Proof of \texorpdfstring{\Cref{prop:finite many clip}}{Proposition \getrefnumber{prop:finite many clip}}}
\label{appendix:proof of prop:finite many clip}
\begin{proof}
According to \Cref{prop:descent}, the sequence $\{\cM_t\}$ is non-increasing, meaning that $\cM_t \leq \cM_0$ for all $t\in\N$. Since $\cE_t \geq 0$, it holds that \begin{equation*}  f(\bar x^t) \leq f(\bar x^t) + \frac{\cE_t}{8\sL}  = \cM_t \leq \cM_0 = f(x^0){}+\frac{\cE_0}{8\sL}. \end{equation*} 
Therefore, the sequence $\{\bar z^t\}=\{\nabla h(\bar x^t)\}\subseteq$ $\overline{\mathrm{lev}_0}:=\{\nabla h(x):f(x) \leq f(x^0){}+\cE_0/(8\sL)\}$, which is bounded according to \Cref{as:bounded level}. Denote $\bar z^t_\theta:=(1-\theta)\bar z^t + \theta \bar z^{t+1}$, then $\{\bar z^t_\theta\}$ is bounded for any $\theta\in[0,1]$. Recall that $H_{\theta_t}^*=\nabla^2 h^*(\bar z^t_{\theta_t})$ with some $\theta_t\in[0,1]$. Fix $\hat z\in\mathrm{Im}(\nabla h)$, set $\widehat{H}^*:=\nabla^2h^*(\hat z)$, and choose $R<\infty$ such that $\|\bar z^t_{\theta_t}-\hat z\|\leq R$ for all $t\in\N$. 

Let $\bar{x}_{\theta_t}:=\nabla h^*(\bar z^t_{\theta_t})$ and $\hat x:=\nabla h^*(\hat z)$. Then,
$\rho_h(\bar{x}_{\theta_t},\hat x)\leq R$. By the HRUC comparison \eqref{eqn:HRUC-Hessian-comparison}, we have
$\nabla^2h(\bar{x}_{\theta_t})
\preceq (1+\zeta(R))\cdot \nabla^2h(\hat x)
$. 
Legendre duality gives $H_{\theta_t}^*=[\nabla^2h(\bar{x}_{\theta_t})]^{-1}$ and $\widehat{H}^*=[\nabla^2h(\hat x)]^{-1}$. Thus, it holds for all $t\in\N$ that $\widehat{H}^*\preceq \big(1+\zeta(R)\big) H_{\theta_t}^*$, implying
\begin{equation}
    \label{eq:prop Hessian bound}
    \mu^* I \preceq H_{\theta_t}^*\qquad \text{where}\qquad \mu^*:=\frac{\lambda_{\min}(\widehat{H}^*)}{1+\zeta(R)}.
\end{equation}
Clearly, $\max_{i\in[m]}\{\|s_i^t\|\}=\delta$ for all $t\in{\cal T}$. Then, it follows from \eqref{eq:prop:descent} and \eqref{eq:prop Hessian bound} that  
$$|{\cal T}|\cdot\frac{\mu^*\delta^2}{12m\eta}\leq \frac{\mu^*}{12\eta}\cdot \sum_{t\in\mathcal{T}}\frac{1}{m}{\sum}_{i=1}^m\|s^t_i\|^2\leq \cM_0-\underline{f}=f(x^0)-\underline{f}{}+\frac{\cE_0}{8\sL},$$
which suggests that $|{\cal T}|\leq \frac{12m\eta}{\mu^*\delta^2}\cdot(f(x^0)-\underline{f}{}+\cE_0/(8\sL))<+\infty$. This characterizes the finite constant $T_{\rm clip}$. 
\end{proof}




\bibliographystyle{siam}
\bibliography{references}

\end{document}